\documentclass[a4paper, reqno]{amsart}
\newif\iffinal
\finaltrue

\pdfoutput=1 

\usepackage[utf8]{inputenc}
\usepackage[T1]{fontenc}
\usepackage{lmodern}
\usepackage[english]{babel}
\usepackage[shortlabels]{enumitem}
\usepackage{amssymb}
\usepackage{mathtools}
\usepackage[pdfencoding=auto]{hyperref}
\usepackage{colortbl}
\usepackage{color}

\usepackage{letltxmacro}
\usepackage[textsize=scriptsize]{todonotes}
\LetLtxMacro\todonotestodo\todo
\renewcommand{\todo}[2][]{\todonotestodo[backgroundcolor=yellow, #1]{TODO: {#2}}}
\iffinal\renewcommand{\todo}[2][]{}\fi

\iffinal
\else
\usepackage[mark]{gitinfo2}
\renewcommand{\gitMark}{\jobname\,\textbullet{}\,\gitFirstTagDescribe\,\textbullet{}\,\gitAuthorName,\,\gitAuthorIsoDate}
\fi
\makeatletter
\newtheorem*{rep@theorem}{\rep@title}
\newcommand{\newreptheorem}[2]{%
\newenvironment{rep#1}[1]{%
 \def\rep@title{#2 \ref{##1}}%
 \begin{rep@theorem}}%
 {\end{rep@theorem}}}
\makeatother

\newreptheorem{theorem}{Theorem}
\newreptheorem{corollary}{Corollary}

\DeclareMathOperator{\fixdiv}{d}

\newtheorem{theorem}{Theorem}

\newtheorem{proposition}{Proposition}[section]
\newtheorem{corollary}[proposition]{Corollary}
\newtheorem{fact}[proposition]{Fact}
\newtheorem{lemma}[proposition]{Lemma}
\theoremstyle{definition}
\newtheorem{definition}[proposition]{Definition}
\newtheorem{notation}[proposition]{Notation}
\newtheorem{convention}[proposition]{Convention}

\newtheorem{remark}[proposition]{Remark}

\DeclareMathOperator{\lcm}{lcm}
\DeclareMathOperator{\Int}{Int}

\DeclareMathOperator{\genfac}{fac}
\DeclarePairedDelimiter{\norm}{\lVert}{\rVert}

\let\assoc=\approx

\newcommand{\IntR}{\Int(R)}

\newcommand{\Rhat}{{\hat R}}
\newcommand{\Mhat}{{\hat M}}
\newcommand{\N}{{\mathbb N}}
\newcommand{\QQ}{{\mathbb Q}}
\newcommand{\Z}{{\mathbb Z}}
\newcommand{\val}{v}

\renewcommand{\vec}[1]{\mathbf{#1}}
\newcommand{\bfE}{\mathbf{E}}
\newcommand{\one}{\mathbf{1}}
\newcommand{\calP}{\mathcal{P}}

\def\C{{\mathcal{B}}}
\def\P{{\mathcal{C}}}
\def\S{{\mathcal{S}}}
\def\PS{{(\P,S)}}
\def\x{{\mathbf{x}}}

\def\F#1{\mathcal{P}(#1)}
\def\Q#1{Q(#1)}
\def\R#1{\mathcal{R}(#1)}
\def\V#1{V(#1)}

\def\M#1{{\mathcal{M}}(#1)}

\title%
[Split absolutely irreducible integer-valued polynomials]
{%
  Split absolutely irreducible integer-valued polynomials over
  discrete valuation domains %
}



\author{Sophie Frisch}
\address[Sophie Frisch]{
  Institute of Analysis and Number Theory \\
  Kopernikusgasse 24/II \\
  8010 Graz \\
  Austria
}
\email{\href{mailto:frisch@math.tugraz.at}{frisch@math.tugraz.at}}
\thanks{S.~Frisch is supported by the Austrian Science Fund (FWF): P~30934}

\author{Sarah Nakato}
\address[Sarah Nakato]{
  Institute of Analysis and Number Theory \\
  Kopernikusgasse 24/II \\
  8010 Graz \\
  Austria
}
\email{\href{mailto:snakato@tugraz.at}{snakato@tugraz.at}}
\thanks{S.~Nakato is supported by the Austrian Science Fund (FWF): P~30934}

\author{Roswitha Rissner}
\address[Roswitha Rissner]{
  Department of Mathematics \\
  University of Klagenfurt\\
  Universitätsstraße 65-67\\
  9020 Klagenfurt am Wörthersee\\
  Austria}
\email{\href{mailto:roswitha.rissner@aau.at}{roswitha.rissner@aau.at}}
\thanks{R.~Rissner is supported by the Austrian Science Fund (FWF): P~28466,
  DOC~78}

\begin{document}

\begin{abstract}
  Regarding non-unique factorization of integer-valued polynomials
  over a discrete valuation domain $(R,M)$ with finite residue field, it is
  known that there exist absolutely irreducible elements, that is,
  irreducible elements all of whose powers factor uniquely, and
  non-absolutely irreducible elements.

  We completely and constructively characterize the absolutely
  irreducible elements among split integer-valued polynomials. They
  correspond bijectively to finite sets, which we call
  \emph{balanced}, characterized by a combinatorial property regarding
  the distribution of their elements among residue classes of powers
  of $M$. For each such balanced set as the set of roots of a split
  polynomial, there exists a unique vector of multiplicities and a
  unique constant so that the corresponding product of monic linear
  factors times the constant is an absolutely irreducible
  integer-valued polynomial.  This also yields sufficient criteria for
  integer-valued polynomials over Dedekind domains to be absolutely
  irreducible.
\end{abstract}

\maketitle


\section{Introduction}
\label{section:intro}
In rings with non-unique factorization into irreducibles there usually
exist irreducible elements some of whose powers have factorizations
into irreducibles other than the obvious one. They are called
non-absolutely irreducible (cf.~Definition~\ref{def:abs-irred}). To
understand patterns of non-unique factorizations it is important to
identify the non-absolutely irreducible elements.  In rings of
integer-valued polynomials
\begin{equation*}
  \IntR = \{f \in K[x] \mid f(R)\subseteq R\},
\end{equation*}
where $R$ is a domain with quotient field $K$, examples of both
absolutely irreducible and non-absolutely irreducible elements have
been given, for instance by the second author of this
paper~\cite{NS2020:Non-Abs}. Note that absolutely irreducible elements
are also called strong atoms, for instance, by~Chapman and
Krause~\cite{ChKr2012:Atomic-decay}, or completely irreducible,
by~Kaczorowski~\cite{Kaczorowski:1981:compl-irred}.  Some of
Nakato's examples concern polynomials that split over $K$.

In this paper, we completely characterize split absolutely irreducible
integer-valued polynomials over a discrete valuation domain $(R,M)$
with finite residue field. It suffices to consider polynomials whose
roots are elements of $R$, as the only split absolutely irreducible
integer-valued polynomials with roots in $K\setminus R$ are linear
polynomials and hence easily understood. We, therefore, consider
polynomials of the form
\begin{equation}\label{eq:genpoly}
  F=c^{-1}f \text{ with } f=\prod_{s\in S} (x-s)^{m_s}
\end{equation}
where $S\subseteq R$ is a finite set and $c\in R$. We show in
Theorem~\ref{characterization_thm} that $F$ is absolutely irreducible
in $\Int(R)$ if and only if $c$ is a generator of the fixed divisor of
$f$, $S$ is a balanced set, and $f$ the equalizing polynomial of
$S$. Balanced refers to the distribution of $S$ among residue classes
of the powers of $M$ (Definition~\ref{definition:balanced}) and the
equalizing polynomial results from a particular choice of
multiplicities $(m_s)_{s\in S}$
(Definition~\ref{definition:equ-poly}), and the fixed divisor is the
ideal generated by the image of $f$ (Definition~\ref{definition:fixdiv}).

So far the study of non-unique factorization has focused on
Krull monoids, which are characterized by having a so-called
``divisor theory''. Among integral domains, Krull domains are
exactly the ones whose multiplicative monoids $D\setminus \{0\}$ are
Krull monoids, cf.~\cite{GeHa2006:NUFactorizations}.

The ring $\IntR$ of integer-valued polynomials over a discrete
valuation domain (with finite residue field) is known to be Prüfer,
but not Krull, cf.~\cite{CahenChabertFrisch:2000:interpolation,
  Loper:1998:intdpruefer}. What is Krull, however, is the monadic
submonoid generated by a single polynomial $f$, that is,
\begin{equation*}
  [\![f]\!] = \{g \in \Int(D) \mid g \text{ divides } f^n
  \text{ for some } n\in \N \}.
\end{equation*}
This submonoid contains all the information about all
factorizations of the powers of a single polynomial
$f$. (Reinhart~\cite{Reinhart:2014:monadic} proved that the monadic
submonoid of $\Int(R)$ is Krull for factorial domains $R$ and
Frisch~\cite{Frisch:2014:monadic} extended this result to Krull
domains $R$.)

Earlier work on factorization-theoretic properties in rings of
integer-valued polynomials can be found in the work of Cahen and
Chabert~\cite{Cahen-Chabert:1995:Elasticity-for-IVP}, Anderson, Cahen,
Chapman and Smith~\cite{AndersonS-Cahen-Chapman-Smith:1995:fac-iv} as
well as Chapman and
McClain~\cite{Chapman-McClain:2005:irred-iv-poly}. For a thorough
introduction into the theory of integer-valued polynomials we refer to the
textbook of Cahen and Chabert~\cite{Cahen-Chabert:1997:book} and their
more recent survey~\cite{Cahen-Chabert:2016:survey}.

Returning to absolute irreducibility, it is immediately seen that
$f$ is absolutely irreducible if and only if the monadic submonoid
$[\![f]\!]$ of $f$ is factorial, which is again equivalent to the
divisor class group of $[\![f]\!]$ being trivial. Our results,
therefore, add to the insight into the monadic submonoids of rings of
integer-valued polynomials.

Moreover, absolutely irreducible elements play an important role when
it comes to the construction of elements with a certain factorization
behaviour. For rings of integer-valued polynomials, only little is
known so far about absolutely irreducible polynomials.

Frisch and Nakato~\cite{FrNa2019:Graphtheoretic} give a
graph-theoretic criterion for an integer-valued polynomial with
square-free denominator over a principal ideal domain to be absolutely
irreducible. For polynomials with squares appearing in the
denominator, the graph-theoretic condition is shown to be sufficient,
but not necessary.  One consequence of their result is that the
binomial polynomial $\binom{x}{p}$ is absolutely irreducible in
$\Int(\Z)$ for a prime number $p$. The latter has been shown before by
McClain~\cite{McClain:2004:honorsthesis}.

In the general case $\binom{x}{n}$ their graph-theoretic argument is
not applicable due to the nature of the denominator. Recently, Rissner
and Windisch~\cite{Rissner-Windisch:2021:binom} verified that the
binomial polynomials $\binom{x}{n}$ are indeed absolutely irreducible
in $\Int(\Z)$ for any $n\in \N$. For the special case where $n=p^k$ is
a prime power, Corollary~\ref{corollary:gen-binom} below serves as an
alternative proof.

Here is a brief outline of our strategy for characterizing split
absolutely irreducible polynomials.  We set aside those that have a
root in $K\setminus R$.  (We will later show that they are all
linear, see Corollary~\ref{corollary:root-set-K-R}). From then on we
only consider polynomials that split over $R$. We define balanced sets
(Definition~\ref{definition:balanced}) and establish a host of somewhat
technical facts about them
(Sections~\ref{section:partition},~\ref{subsection:small-fat-sets}). These facts
allow us to find for each balanced set $S$ a vector of multiplicities
$(m_s)_{s\in S}$ (Proposition~\ref{Bpositive_lemma}) such that
$\prod_{s\in S}(x-s)^{m_s}$ multiplied by an appropriate constant is
absolutely irreducible (Theorem~\ref{theorem:equpoly-exists}). We show
that this vector of multiplicities is unique because a certain type of
matrix is non-singular
(Proposition~\ref{proposition:detneq0}). Finally, in
Theorem~\ref{characterization_thm}, we prove that all split absolutely
irreducible polynomials with roots in $R$ are of this kind by showing
that the set of roots always is a balanced set.

\section{Preliminaries}
\label{sec:preliminaries}

\begin{convention}\label{conventionRD}
  Throughout this paper, unless explicitly stated otherwise, $D$ is
  always a domain with quotient field $K$ and $(R,M)$ always denotes a
  discrete valuation domain with quotient field $K$ and finite residue
  field.
\end{convention}

\subsection{Integer-valued polynomials}\label{subsection:prel-iv}
Note that we only introduce the basic notions concerning
integer-valued polynomials required in this work. A detailed treatment
of the theory of integer-valued polynomials can be found in the
textbook of Cahen and Chabert~\cite{Cahen-Chabert:1997:book} and their
recent survey~\cite{Cahen-Chabert:2016:survey}.

\begin{definition}\label{definition:fixdiv}
  Let $D$ be a domain with quotient field $K$. The \emph{ring of
    integer-valued polynomials on $D$} is defined as
  \begin{equation*}
    \Int(D) = \{F \in K[x] \mid F(D) \subseteq D\}.
  \end{equation*}

  For $F\in \Int(D)$, the \emph{fixed divisor} of $F$ is the ideal
  \begin{equation*}
    \fixdiv(F) = (F(a) \mid a \in D)
  \end{equation*}
  of $D$ generated by the elements $F(a)$ with $a\in D$.

  The polynomial $F$ is said to be \emph{image-primitive} if
  $\fixdiv(F) = D$.
\end{definition}

\begin{remark}
  Let $f\in D[x]$ and $b\in D$. Then
  \begin{equation*}
    \frac{f}{b}\in \Int(D) \;\Longleftrightarrow\; \fixdiv(f) \subseteq bD
  \end{equation*}
\end{remark}

\begin{remark}\label{remark:intd-fixdiv}
  Let $D$ be a domain and $g\in D[x]$. If $P$ is a prime ideal of $D$
  such that $\fixdiv(g) \subseteq P$, then $g\in P[x]$ or
  $\norm{P} = |D/P| \le \deg(g)$.
\end{remark}

\begin{remark}\label{remark:irred-img-prim}
  Let $D$ be a principal ideal domain with quotient field $K$.
  \begin{enumerate}
  \item All fixed divisors are principal ideals. Below, the notation
    $\fixdiv_f$ stands for an arbitrary but fixed generator of the
    fixed divisor of $f$.
  \item There is for every $f\in D[x]$ an image-primitive polynomial
    $F\in\Int(D)$ (unique up to multiplication by units of $D$)
    associated to $f$ in $K[x]$, namely, $F = \fixdiv_{f}^{-1}f$,
    where $\fixdiv_{f}$ is a generator of the fixed divisor of $f$.
  \item An irreducible polynomial in $\Int(D)$ is necessarily
    image-primitive.
  \item In the equivalence class of $f$ with respect to multiplication
    by non-zero constants in $K$, it is only the image-primitive
    elements that have a chance of being irreducible, or absolutely
    irreducible.  We investigate for which $f\in D[x]$ an
    image-primitive $F\in\Int(D)$ associated to $f$ in $K[x]$ is
    absolutely irreducible.
  \end{enumerate}
\end{remark}

Next, we remind the reader of the prime spectrum of $\IntR$ where $R$
is a discrete valuation domain and introduce some notation.

\begin{notation}\label{notation:primes}
  Let $(R,M)$ be as in Convention~\ref{conventionRD}. Further, let
  $\Rhat$ denote the $M$-adic completion of $R$ and $\Mhat$ its
  maximal ideal.
  \begin{enumerate}
  \item For a monic, irreducible polynomial $q\in K[x]$, we write
    \[
      Q_q = \IntR \cap qK[x].
    \]
  \item For $\alpha\in \Rhat$, we write
    \[
      M_\alpha = \{G\in \IntR \mid G(\alpha)\in \Mhat \}.
    \]
    Here, $G(\alpha)$ is defined by extending the $M$-adically
    uniformly continuous function $G\colon R\rightarrow R$ uniquely to
    $\Rhat$.
  \item Note that, for $a\in R$,
    \begin{equation*}
      M_a = \{G\in \IntR \mid G(a)\in M \}
    \end{equation*}
    because  $\Mhat \cap K = M$.
  \item For $F\in \IntR$, we write
    \begin{align*}
      \M F &= \{ M_a \mid a\in R,\, F\in M_a\} \text{ and } \\
      \Q F &= \{ Q_q \mid q \text{ monic, irreducible in }
      K[x],\, F\in Q_q\}.
    \end{align*}
  \end{enumerate}
\end{notation}

\begin{fact}[{\cite[Corollary~V.1.2, Lemma~V.1.3, Proposition~V.2.2]{Cahen-Chabert:1997:book}}]\label{fact:iv-prime-ideals}
  Let $(R,M)$ be as in Convention~\ref{conventionRD}. Further, let
  $\Rhat$ denote the $M$-adic completion of $R$ and $\Mhat$ its
  maximal ideal.

  Then
  \[
    \mathrm{Spec}(\IntR)= \{(0)\}\cup \{ Q_q \mid q\in K[x] \text{
      monic, irreducible } \} \cup \{ M_\alpha \mid \alpha\in \Rhat \}.
  \]
  \end{fact}

\subsection{Absolute irreducibility and prime separation}
\label{sec:prime-separation}
Again, we only introduce the notions and tools from factorization
theory that we require below. For a thorough treatment of the topic,
we refer to the textbook of Geroldinger and
Halter-Koch~\cite{GeHa2006:NUFactorizations}.

\begin{convention}
  Ring theoretic entities such as units, irreducible elements, prime
  ideals, etc.\ are defined with respect to the ring $\IntR$, unless
  specified otherwise. Similarly, principal ideals: $(F)$ means
  $F\IntR$ unless specified otherwise.
\end{convention}

\begin{convention}
  As usual in factorization theory, we do not distinguish between
  associated elements of $\IntR$, that is, elements that differ only
  by multiplication by a unit of $\IntR$. (Recall that the units of
  $\IntR$ are the units of $R$.) Also, regarding uniqueness,
  \emph{essentially unique} means unique up to multiplication by units of
  $\IntR$.
\end{convention}

\begin{definition}\label{def:abs-irred}
  Let $D$ be a domain. An irreducible element $d\in D$ is called
  \emph{absolutely irreducible} if $d^n$ factors uniquely in $D$ for
  all integers $n\ge 1$.
\end{definition}

Chapman and Krause~\cite[Lemma~2.1]{ChKr2012:Atomic-decay} showed for
an irreducible element $c$ of an atomic domain $D$ the equivalence of
the following two statements:
\begin{enumerate}
\item $c$ is absolutely irreducible.
\item For all irreducible $b\in D$ which are not associated to $c$
  there exists a prime ideal $P$ of $D$ such that $b \in P$ and
  $c \not \in P$.
\end{enumerate}

We rewrite and refine this characterization below in
Proposition~\ref{proposition:not-abs-irr_lemma}. The proposition's
detailed formulation in particular points out which prime ideals we
need to consider in our further work.

\begin{definition}
  For an ideal $I$ of a domain $D$, let
  \begin{equation*}
    \V I = \{P \in \operatorname{Spec}(D) \mid I \subseteq P\}
  \end{equation*}
  denote the subset of the prime spectrum of $D$ of prime ideals
  containing $I$. For a principal ideal $cD$, we write $\V c$ for
  $\V {cD}$.
\end{definition}

Recall that the radical $\sqrt{I}$ of an ideal $I$ is defined by
\begin{equation*}
  \sqrt{I} = \{d \in D \mid d^k\in I \text{ for some } k\in \N\}
\end{equation*}
and, as we all know, this is equivalent to
\begin{equation*}
  \sqrt{I} = \bigcap_{P\in \V I}P.
\end{equation*}
Therefore, for ideals $I$ and $J$, $\sqrt{I} \subseteq \sqrt{J}$ if
and only if $\V I \supseteq \V J$. More generally, if for some class
of ideals (e.g.\ finitely generated ideals), the radical is always an
intersection of prime ideals from some special subset
$\S \subseteq \operatorname{Spec}(D)$, then again
$\sqrt{I } \subseteq \sqrt{J}$ if and only if
$\V I \cap \S \supseteq \V J \cap \S$.

This simple fact leads to the following extended version of the
criterion of Chapman and Krause above.

\begin{proposition}\label{proposition:not-abs-irr_lemma}
  Let $D$ be an atomic domain and $\S\subseteq \operatorname{Spec}(D)$
  such that for every principal ideal $cD$, the radical $\sqrt{cD}$ is
  the intersection of all $P\in \S$ containing~$c$.
 
  For a non-zero non-unit $d\in D$, the following are equivalent.
  \begin{enumerate}
  \item $d$ is \textbf{not} absolutely irreducible, i.e., some power
    $d^n$ has a factorization into irreducibles essentially different
    from $d\cdots d$ ($n$ copies of $d$).
  \item Some power of $d$ is divisible by some irreducible
    $c\in D$ that is not associated to $d$.
  \item Some power of $d$ is divisible by some non-unit $c\in D$ that
    is not associated to a power of $d$.
  \item $d\in \sqrt{cD}$ for some non-unit $c\in D$ that is not
    associated to a power of $d$.
  \item $\sqrt{dD}\subseteq \sqrt{cD}$ for some non-unit $c\in D$ that
    is not associated to a power of $d$.
  \item $\V c \subseteq \V d$ for some non-unit $c\in D$ that is not
    associated to a power of $d$.
  \item $\V c\cap\S \subseteq \V d\cap\S$ for some non-unit $c\in D$ that
    is not associated to a power of $d$.
\end{enumerate}
\end{proposition}

\begin{proof}
  $(1) \Rightarrow (2)$: Suppose $d^n$ has a factorization
  essentially different from $d\cdots d$ ($n$ copies of $d$). Then,
  since $d^n = d^m$ for $n\neq m$ is impossible by cancellation in a
  domain, there exists an irreducible $c$ not associated to $d$
  dividing $d^n$.
  
  $(2) \Rightarrow (1)$: Suppose $d^n = ce$ with $c$ irreducible and
  not associated to $d$ and $e\in D$. Then, since $D$ is atomic, we
  can use a factorization $e=e_1\cdots e_m$ into irreducibles to get a
  factorization $d^n = c e_1\cdots e_m$ other than $d\cdots d$ ($n$
  copies of $d$).

  For the remaining statements, the equivalence of each statement to the
  preceding one follows from elementary considerations.
\end{proof}

Now we use specific information about the spectrum of $\IntR$ to
rephrase the proposition above for our setting.

\begin{remark}\label{spec_remark}
  Let $(R,M)$ be as in Convention~\ref{conventionRD}. We know that the
  radical $\sqrt{I}$ of every finitely generated ideal $I$ of $\IntR$
  is an intersection of maximal ideals of the special form
  \begin{equation*}
   M_a =  \{G\in \IntR\mid G(a)\in M \}
  \end{equation*}
  for $a\in R$ and prime ideals of the form
  \begin{equation*}
   Q_q =  \IntR \cap qK[x]
  \end{equation*}
  for monic, irreducible polynomials $q\in K[x]$. The maximal ideals
  of the form $M_\alpha$ for $\alpha\in\Rhat\setminus R$ are
  redundant, for reasons of $M$-adic continuity.
\end{remark}

Setting $\S = \M f \cup \Q f$ in
Proposition~\ref{proposition:not-abs-irr_lemma}(7) gives the
following.

\begin{corollary}\label{general-abs-irred_cor}
  Let $(R,M)$ be as in Convention~\ref{conventionRD}, and $F\in \IntR$
  a non-zero non-unit, and $\M F$ and $\Q F$ as in Notation~\ref{notation:primes}(4).

  \begin{enumerate}
  \item $F$ is absolutely irreducible in $\IntR$ if and only if every
    non-unit $G\in \IntR$ satisfying
    $\M G \cup \Q G \subseteq \M F \cup \Q F$ is associated to a power
    of $F$.
  \item If there exists a non-unit $G\in\IntR$ such that
    \begin{equation*}
      \M G \cup \Q G \subsetneq \M F \cup \Q F,
    \end{equation*}
    then $F$ is not absolutely irreducible.
  \end{enumerate}
\end{corollary}

\section{The posh set of a polynomial}\label{sec:fat-set}

Let $(R,M)$ be as in Convention~\ref{conventionRD}. We will
characterize absolutely irreducible polynomials that split over $R$ in
terms of their root sets. It is, therefore, convenient to work with a
product of linear factors $f = \prod_{s\in S} (x-s)^{m_s}$ with
$S\subseteq R$ and investigate whether the essentially unique
image-primitive $F = \fixdiv_{f}^{-1}f$ which is $K[x]$-associated to
$f$ is absolutely irreducible, cf.~Remark~\ref{remark:irred-img-prim}.

We now reformulate the criterion in
Corollary~\ref{general-abs-irred_cor} for absolute irreducibility in
terms of $f$ instead of $F$.

It is easily seen that $\Q {F} = \Q {f}$. The set $\M F$, however, is
not invariant under multiplication with constants in $K$. To resolve
this issue we now introduce a set $\F F$ that is in one-to-one correspondence
with $\M F$ and is invariant under multiplication with constants in
$K$, so that $\F F = \F f$.

\begin{definition}\label{fatdef}
  Let $(R,M)$ be as in Convention~\ref{conventionRD}.  For
  $F\in K[x]$, we denote by $\F F$ the \emph{posh set} of $F$, that
  is,
  \[
    \F F=\{ r\in R\mid \val(F(r)) > \min_{t\in R} \val(F(t)) \}.
  \]
\end{definition}

\begin{remark}\label{remark:fat-fixdiv}\label{remark:Mf-imgpr}
  Let $(R,M)$ be as in Convention~\ref{conventionRD}. 
  \begin{enumerate}
  \item If $F\in \IntR$ and $\fixdiv_F$ a generator of the fixed
    divisor of $F$ (cf.~Definition~\ref{definition:fixdiv}) then
    \begin{equation*}
      \min_{t\in R} \val(F(t)) = \val(\fixdiv_F).
    \end{equation*}
  \item A polynomial $F\in \IntR$ is
    image-primitive if and only if there exists an element $a\in R$
    such that $F(a) \not\in M$. Therefore, $F$ is
    image-primitive if and only if there exists $a\in R$ with
    $M_a\notin \M F$.
  \end{enumerate}
\end{remark}

\begin{remark}\label{fatrelevance_remark}
  Let $F\in \IntR$ be such that $\min_{r\in R}\val(F(r))=0$, that is,
  $F$ image-primitive. Then 
  \[
    a\in \F F \quad\Longleftrightarrow\quad F(a)\in M.
  \]
  There is, therefore, a one-to-one correspondence between $\F F$ and
  $\M F$, given by $a\mapsto M_a$, in other words,
  \[
    \M F = \{M_a\mid a\in\F F\}\quad\textrm{and}\quad \F F = \{a\in
    R\mid M_a\in \M F\}.
  \]
  Now, unlike $\M F$, the posh set of $F$ is invariant under
  multiplication of $F$ by non-zero constants in $K$. Suppose
  $f\in K[x]$ and $F\in\IntR$ such that $F$ is associated to $f$ in
  $K[x]$, that is, $FK[x] = fK[x]$. If $F$ is image-primitive then
  \[
    \M F = \{M_a\mid a\in\F{f}\} \quad\text{ and }\quad \F {f} = \{a\in
    R\mid M_a\in \M F\} = \F F.
  \]
\end{remark}

We can, therefore, let the posh set of $f$ stand in for $\M{F}$ when we
formulate a criterion for an image-primitive polynomial $F$ to be
absolutely irreducible. Before we formulate this criterion below in
Proposition~\ref{proposition:special-abs-irr}, we establish that by
switching between $F$ and the uniquely determined monic polynomial $f$
which is $K[x]$-associated to $F$ we keep the information about the
powers of these polynomials.

\begin{lemma}\label{power_lemma}
  Let $(R,M)$ be as in Convention~\ref{conventionRD}.  Let $f$,
  $g\in K[x]$ be monic and $F$, $G\in\IntR$ be image-primitive such
  that $FK[x]=fK[x]$ and $GK[x]=gK[x]$.

  Then
  \[
    g = {f}^n \quad\Longleftrightarrow\quad G \assoc F^n,
  \]
  where $\assoc$ means associated in $\IntR$, that is, differing only
  by multiplication by a unit of $R$.
\end{lemma}

\begin{proof}
  Necessarily $G = c^{-1} g$ and $F = d^{-1} f$ for generators $c$ and
  $d$ of the fixed divisors of $g$ and $f$, respectively.  If $g=f^n$
  then $c\assoc d^n$ and, hence,
  $G = c^{-1} g \assoc d^{-n} {f}^n = F^n$.

  Conversely, if $G \assoc F^n$, then $g$ equals ${f}^n$ because each
  is the unique monic generator of the ideal $GK[x]=F^nK[x]$.
\end{proof}

\begin{proposition}\label{proposition:special-abs-irr}
  Let $(R,M)$ be as in Convention~\ref{conventionRD}.  Let $f\in K[x]$
  be monic and $F\in\IntR$ be image-primitive such that $FK[x]=fK[x]$.

  Then the following assertions are equivalent:
  \begin{enumerate}
  \item $F$ is absolutely irreducible.
  \item Every monic $g\in K[x]$ with $g\neq 1$ satisfying
    $\Q {g}\subseteq \Q{f}$ and $\F {g}\subseteq \F{f}$ is a power of
    $f$.
  \end{enumerate}
\end{proposition}

\begin{proof}
  Assume $(1)$.
  Let
  $G\in\IntR$ be image-primitive with $GK[x]=gK[x]$.  Then
  $\Q G = \Q {g} \subseteq \Q {f} = \Q F$ and
  \begin{equation*}
    \M G = \{M_a\mid a\in \F{g}\}\subseteq \{M_a\mid a\in \F{f}\}= \M
    F.
  \end{equation*}
  By Corollary~\ref{general-abs-irred_cor}, $G\assoc F^n$ for some
  $n\in\N$, and, hence, $g={f}^n$, by Lemma~\ref{power_lemma}.

  Conversely, assume $(2)$.
  Now, if some $G\in\IntR$ satisfies $\M G \subseteq \M F$ and
  $\Q G \subseteq \Q F$ then, since
  $\M F\subsetneq \{M_a\mid a\in R\}$, it follows that $G$ is
  image-primitive, cf.~Remark~\ref{remark:Mf-imgpr}.  Let $g$ be monic
  with $gK[x]=GK[x]$. Then
  \begin{equation*}
    \F {g} = \{a\in R\mid M_a\in \M G\} \subseteq \{a\in R\mid M_a\in
    \M F\} = \F {f}
  \end{equation*}
  and
  \begin{equation*}
    \Q {g} = \Q G \subseteq \Q F = \Q {f}.
  \end{equation*}
  By hypothesis, $g ={f}^n$ (for some $n$) follows, and, by
  Lemma~\ref{power_lemma}, $G\assoc F^n$.  Applying
  Corollary~\ref{general-abs-irred_cor}, we conclude that $F$ is
  absolutely irreducible.
\end{proof}

\begin{corollary}\label{strict}
  Let $(R,M)$ be as in Convention~\ref{conventionRD}, $f\in K[x]$ be
  monic and $F\in\IntR$ image-primitive, such that $fK[x]=FK[x]$.

  If there exists a monic, non-constant polynomial $g\in K[x]$ with
  \[
    \Q {g}\subseteq \Q{f} \quad\textrm{and}\quad \F {g}\subsetneq
    \F{f}
  \]
  or
  \[
    \Q {g}\subsetneq \Q{f} \quad\textrm{and}\quad \F {g}\subseteq
    \F{f}
  \]
  then $F$ is not absolutely irreducible.
\end{corollary}


We now turn our attention to split integer-valued polynomials, which
are our main subject of investigation.

\begin{remark}\label{remark:split-and-Q}
  Let $f = \prod_{s\in S}(x-s)^{m_s}$ with $m_s>0$ positive integers
  for $s\in S$.  For a monic, non-constant polynomial $g\in K[x]$ the
  condition
  \begin{equation}\label{eq:g0split}
    \Q {g} \subseteq \Q{f} 
  \end{equation}
  is equivalent to $g = \prod_{t\in T}(x-t)^{k_t}$ with
  $\emptyset \neq T \subseteq S$ and $k_t>0$ for $t\in T$.
\end{remark}

Of the two entities considered in
Proposition~\ref{proposition:special-abs-irr} and
Corollary~\ref{strict}, $\Q{f}$ is completely determined by the
irreducible factors of $f$, that is, in the case of a split
polynomial, by the roots (no multiplicities considered).  The posh set
$\F {f}$ is also closely connected to the root set but here the
multiplicities matter, as we shall see.

\section{Distribution of the roots of a split polynomial}
\label{section:partition}

In this section we discuss the connection between the posh set of a
split polynomial $f = \prod_{s\in S}(x-s)^{m_s}$ and the distribution
of the root set $S$ among the residue classes of powers of the maximal
ideal $M$ of the discrete valuation domain $R$.

\begin{definition}\label{def:Mpart-bal}
  Let $(R,M)$ be as in Convention~\ref{conventionRD}.
  \begin{enumerate}
  \item\label{item:Mpart} By an \emph{$M$-adic partition of $R$} we
    understand a finite partition of $R$ into residue classes of
    powers of $M$, that is
    \[
      \P = \{ s+M^{n_s}\mid s\in S\},
    \]
    where $S\subseteq R$ is a finite set,
    $R=\bigcup_{s\in S} (s+M^{n_s})$ and
    $(s+M^{n_s})\cap (t+M^{n_t})=\emptyset$ for $s\ne t$. We say the
    set $S$ is a \emph{set of representatives of $\P$}.
  \item The pair $(S,\P)$ where $\P$ is an $M$-adic partition and $S$
    is a set of representatives is also called a \emph{pointed
      $M$-adic partition of $R$}.
\end{enumerate}
\end{definition}

\begin{definition}\label{definition:balanced}
  Let $(R,M)$ be as in Convention~\ref{conventionRD}.  We call
  $S\subseteq R$ \emph{$(R,M)$-balanced} if, when we take for each
  $s\in S$ the minimal $n_s$ such that $s+M^{n_s}$ contains no other
  element of $S$, the resulting disjoint basic $M$-adic neighborhoods
  $s+M^{n_s}$ cover $R$. If $R$ and $M$ are understood we just say
  \emph{balanced} for $(R,M)$-balanced.
\end{definition}

\begin{remark}\label{remark:balanced-partition}
  Note that the set of representatives $S$ of a pointed $M$-adic
  partition $\PS$ is a balanced set; and conversely, that for each
  balanced set $S$ the maximal basic $M$-adic neighborhoods
  $s+M^{n_s}$ disjoint from the remaining elements of $S$, together
  with $S$ as a system of representatives, constitute a pointed
  $M$-adic partition of $R$.

  The blocks of an $M$-adic partition can be visualized as the leaves
  of a $q$-adic tree where $q = |R/M|$. We obtain a balanced set by
  choosing one representative for each block.  For example,
  Figure~\ref{fig:tree} shows the tree corresponding to the balanced
  set $\{0,3,6,15,24,1, 2,11,20,47,128,209,74,5,8\}$ in $\Z$ localized
  at~$3$.
\end{remark}
\begin{figure}[h]\centering
    \resizebox{0.8\textwidth}{!}{
    \begin{tikzpicture}
      \node[fill, circle, inner sep = 1pt,] (root) at (0,0) {};
      \node[black] (leveltop) at (-5.5, -0.2) {\footnotesize $\pmod{3^0}$};
      \node  at (5.5, 0) {\footnotesize \phantom{$\pmod{3^0}$}};
      \node[black, below = 0.5cm of leveltop] (level1) {\footnotesize $\pmod{3^1}$};
      \node[black, below = 0.5cm of level1] (level2) {\footnotesize $\pmod{3^2}$};
      \node[black, below = 0.5cm of level2]  (level3){\footnotesize $\pmod{3^3}$};
      \node[black, below = 0.5cm of level3]  (level4){\footnotesize $\pmod{3^4}$};
      \node[black, below = 0.5cm of level4]  (level5){\footnotesize $\pmod{3^5}$};
    
      \node[below = of root,
            xshift = -3cm, fill, circle, inner sep = 1pt, label=left:{\footnotesize $0$}] (03) {};
      \node[below = of root, fill, circle, inner sep = 1pt, label=left: {\footnotesize $1$}] (13){};
      \node[below = of root, xshift =  3cm, fill, circle, inner sep = 1pt, label=right: {\footnotesize $2$}] (23) {};
      \draw (root) -- (03);    
      \draw (root) -- (13);    
      \draw (root) -- (23);
    
      \node[below = of 03, xshift = -1cm, fill, circle, inner sep = 1pt,
            label=left:{ \footnotesize $0$}] (09) {};
      \node[below = of 03, xshift = 0cm, fill, circle, inner sep = 1pt,
            label=left:{ \footnotesize $3$}] (39) {};
      \node[below = of 03, xshift = 1cm, fill, circle, inner sep = 1pt,
            label=right:{ \footnotesize $6$}] (69) {};
      \draw (03) -- (69);
      \draw (03) -- (09);
      \draw (03) -- (39);

      \node[below = of 69, xshift = -1cm, fill, circle, inner sep = 1pt,
            label=left:{ \footnotesize $6$}] (627) {};
      \node[below = of 69, xshift = 0cm, fill, circle, inner sep = 1pt,
            label=left:{ \footnotesize $15$}] (1527) {};
      \node[below = of 69, xshift = 1cm, fill, circle, inner sep = 1pt,
            label=right:{ \footnotesize $24$}] (2427) {};
      \draw (69) -- (627);
      \draw (69) -- (1527);
      \draw (69) -- (2427);
      
      \node[below = of 23, xshift = -1cm, fill, circle, inner sep = 1pt,
            label=left:{ \footnotesize $2$}] (29) {};
      \node[below = of 23, xshift = 0cm, fill, circle, inner sep = 1pt,
            label=left:{ \footnotesize $5$}] (59) {};
      \node[below = of 23, xshift = 1cm, fill, circle, inner sep = 1pt,
            label=right:{ \footnotesize $8$}] (89) {};
      \draw (23) -- (29);
      \draw (23) -- (59);
      \draw (23) -- (89);

      \node[below = of 29, xshift = -1cm, fill, circle, inner sep = 1pt,
            label=left:{ \footnotesize $2$}] (227) {};
      \node[below = of 29, xshift = 0cm, fill, circle, inner sep = 1pt,
            label=left:{ \footnotesize $11$}] (1127) {};
      \node[below = of 29, xshift = 1cm, fill, circle, inner sep = 1pt,
            label=right:{ \footnotesize $20$}] (2027) {};
      \draw (29) -- (227);
      \draw (29) -- (1127);
      \draw (29) -- (2027);

      \node[below = of 2027, xshift = -1cm, fill, circle, inner sep = 1pt,
            label=left:{ \footnotesize $20$}] (203^4) {};
      \node[below = of 2027, xshift = 0cm, fill, circle, inner sep = 1pt,
            label=left:{ \footnotesize $47$}] (473^4) {};
      \node[below = of 2027, xshift = 1cm, fill, circle, inner sep = 1pt,
            label=right:{ \footnotesize $74$}] (743^4) {};
      \draw (2027) -- (203^4);
      \draw (2027) -- (473^4);
      \draw (2027) -- (743^4);
      
      \node[below = of 473^4, xshift = -1cm, fill, circle, inner sep = 1pt,
            label=left:{ \footnotesize $47$}] (473^5) {};
      \node[below = of 473^4, xshift = 0cm, fill, circle, inner sep = 1pt,
            label=left:{ \footnotesize $128$}] (1283^5) {};
      \node[below = of 473^4, xshift = 1cm, fill, circle, inner sep = 1pt,
            label=right:{ \footnotesize $209$}] (2093^5) {};
      \draw (473^4) -- (473^5);
      \draw (473^4) -- (1283^5);
      \draw (473^4) -- (2093^5);

     \node[draw=black, fill=white, inner sep = 4pt] at (13) {};
     \node[draw=black, fill=white, inner sep = 4pt] at (09) {};
     \node[draw=black, fill=white, inner sep = 4pt] at (39) {};
     \node[draw=black, fill=white, inner sep = 4pt] at (59) {};
     \node[draw=black, fill=white, inner sep = 4pt] at (89) {};
     \node[draw=black, fill=white, inner sep = 4pt] at (627) {};
     \node[draw=black, fill=white, inner sep = 4pt] at (1527) {};
     \node[draw=black, fill=white, inner sep = 4pt] at (2427) {};
     \node[draw=black, fill=white, inner sep = 4pt] at (227) {};
     \node[draw=black, fill=white, inner sep = 4pt] at (1127) {};
     \node[draw=black, fill=white, inner sep = 4pt] at (203^4) {};
     \node[draw=black, fill=white, inner sep = 4pt] at (743^4) {};
     \node[draw=black, fill=white, inner sep = 4pt] at (473^5) {};
     \node[draw=black, fill=white, inner sep = 4pt] at (1283^5) {};
     \node[draw=black, fill=white, inner sep = 4pt] at (2093^5) {};
   \end{tikzpicture}
  }
  \caption{}
  \label{fig:tree}
\end{figure}
Clearly not every finite set is balanced. We can, nevertheless, extend
the argument of Remark~\ref{remark:balanced-partition} to associate a
unique $M$-adic partition to every finite subset $S$ of $R$ in such a
way that $S$ contains a balanced subset associated to the same
partition.

\begin{lemma}\label{lemma:S-partition}
  Let $(R,M)$ be as in Convention~\ref{conventionRD} and
  $S\subseteq R$ a finite subset of $R$.

  Then there exists a uniquely determined $M$-adic partition
  \begin{equation*}
    \P_S =  \left\{ s + M^{n_s} \mid s\in S \right\}
  \end{equation*}
  of $R$ such that every residue class $s+M^{n_s}$ that occurs as a
  block of $\P_S$ contains both a residue class of $M^{n_s+1}$
  intersecting $S$ and a residue class of $M^{n_s+1}$ disjoint from
  $S$.

  As a consequence, $S$ contains a balanced subset $S'$ with
  $\P_S = \P_{S'}$.
\end{lemma}

\begin{proof}
  To see the existence of $\P_S$, we construct it, inductively, as
  follows. In each step~$k$, we start with two sets of $M$-adic
  neighborhoods: $\P_k$ (containing residue classes of $M^n$ for
  various $n<k$ that have already been chosen as blocks of our
  partition), and ${\C}_k$ (containing residue classes of $M^k$ under
  consideration as potential blocks of the partition), such that
  \begin{enumerate}
  \item ${\P_k}\uplus {\C_k}$ is a partition of $R$ (we use $\uplus$
    to denote disjoint unions)
  \item each $(r+M^n)\in \P_k$ contains both a residue class of
    $M^{n+1}$ intersecting $S$ and a residue class of $M^{n+1}$
    disjoint from $S$
  \item $(r+M^k)\cap S\ne\emptyset$ for each $(r+M^k)\in \C_k$.
  \end{enumerate}

  The process terminates when $\C_k=\emptyset$ and, consequently,
  $\P_k=\P_S$ is a partition of $R$ with the desired properties.

  At the beginning of Step~$0$, ${\C_0}= \{R\}$ and
  ${\P_0}= \emptyset$.  If some residue classes of $M$ contain
  elements of $S$ and some do not, then we set ${\P}_1=\{R\}$ and
  ${\C_1}=\emptyset$ and we are done.

  Otherwise, we break up $R$ into residue classes of $M$ and put
  these in ${\C}_1$, so that we have ${\P}_1=\emptyset$ and
  ${\C}_1=\{r_1+M,\ldots,r_q+M\}$.

  At Step $k$, we define $\P_{k+1}$ to be the union of $\P_k$ and the
  set of those residue classes of $M^k$ in $\C_k$ that contain
  residue classes of $M^{k+1}$ intersecting $S$ as well as residue
  classes of $M^{k+1}$ disjoint from $S$.  The remaining residue
  classes of $M^k$ in $\C_k$, containing only such residue classes
  of $M^{k+1}$ that intersect $S$ nontrivially, we split into residue
  classes of $M^{k+1}$ and let $\C_{k+1}$ be the set of these
  residue classes.

  Since $S$ is finite, the process terminates and we get the desired
  partition of $R$.

  To see uniqueness, consider that each residue class $r+M^k$ that we
  add to $\P_k$, and, eventually, to $\P_S$, in Step~$k$ must occur as a
  block of the partition, because otherwise $\P_S$ would have to contain
  as blocks some residue classes contained in $r+M^k$ that are
  disjoint from $S$.
\end{proof}

It is a key fact in our characterization of split absolutely
irreducible polynomials that their posh set is as small as can be. To
formalize this, we introduce the rich set of a split polynomial, which
is always contained in its posh set, and show that the posh set equals
the rich set in the case of an absolutely irreducible split
polynomial.

\begin{definition}\label{def:assoc-part-rich-poor}
  Let $(R,M)$ be as in Convention~\ref{conventionRD} and
  $S \subseteq R$ a finite subset.
  \begin{enumerate}
  \item We call the (uniquely determined) partition $\P_S$ of
    Lemma~\ref{lemma:S-partition} the \emph{partition associated to
      $S$}.
  \item For $s\in S$, let $\rho_S(s) \in \N$ be the uniquely
    determined non-negative integer such that $s + M^{\rho_S(s)}$ is a
    partition block of $\P_S$.
  \item An \emph{$S$-rich neighborhood} is a residue class
    $s+M^{\rho_S(s)+1}$ with $s\in S$.
  \item An \emph{$S$-poor neighborhood} is a residue class of the form
    $r + M^{\rho_S(s)+1}$ disjoint from $S$ where
    $r\in s + M^{\rho_S(s)}$ for some $s\in S$.
  \item The {\em{rich set of $S$}}, denoted by $\R S$, is the union of
    the rich neighborhoods of the partition $\P$ associated to $S$,
    that is
    \begin{equation*}
      \R S = \bigcup_{s\in S} s + M^{\rho_S(s) + 1}.
    \end{equation*}
  \item For a polynomial $f\in K[x]$ that splits over $R$ the
    {\em{rich set of $f$}}, denoted by $\R f$, is defined to be the
    rich set of the set of its roots.
  \end{enumerate}
\end{definition}

\begin{remark}\label{remark:rho}
  \begin{enumerate}
  \item It follows from the proof of Lemma~\ref{lemma:S-partition}
    that, for $s\in S$, $\rho_S(s) \in \N_0$ is the minimal number
    such that $s + M^{\rho_S(s)}$ contains an $S$-rich neighborhood as
    well as an $S$-poor neighborhood. In particular, if
    $T\subseteq S$ is a subset of $S$, then $\rho_T(t) \le \rho_S(t)$
    for all $t\in T$.
  \item For $n\in \N$, the following implication holds
    \begin{equation*}
      n\le \rho_S(s) \Longrightarrow s + M^n \nsubseteq \R S.
    \end{equation*}
    In particular, if $T\subseteq S$ and there exists $t\in T$ with
    $\rho_T(t) < \rho_S(t)$, then $\R T \nsubseteq \R S$.
  \item\label{item:disj} If $S$ is a balanced set then the rich set
    \begin{equation*}
      \R S = \biguplus_{s\in S} s + M^{\rho_S(s)+1}
    \end{equation*}
    is the disjoint union of the $S$-rich neighborhoods
    $s + M^{\rho_S(s)+1}$.
  \item If $S'$ is a balanced set contained in $S$ with
    $\P_S = \P_{S'}$, then $\rho_{S'}(s) = \rho_S(s)$ for all
    $s\in S'$ and $\R{S'} \subseteq \R{S}$. Note that equality may
    hold even if $S$ is not balanced as an $S$-rich neighborhood can
    contain more than one element of $S$.
  \end{enumerate}
\end{remark}

\begin{definition}\label{definition:assoc-balanced}
  Let $(R,M)$ be as in Convention~\ref{conventionRD} and
  $S \subseteq R$ a finite subset.  We call a set $S'$ a
  \emph{balanced set associated to $S$} if $S'$ is balanced and
  contained in $S$ with $\P_S = \P_{S'}$. Further, for $t\in S'$, we
  write
  \begin{equation*}
    S_t = \{s \in S \mid s + M^{\rho_S(s)} = t + M^{\rho_S(t)} \}
  \end{equation*}
  for the set of all elements in $S$ which are elements of the
  partition block of $t\in S'$ of the partition $\P_S = \P_{S'}$.
\end{definition}

We are ready to show that the rich set is always contained in the posh
set of a split polynomial.
Let $f\in K[x]$ be a split monic polynomial whose root set $S$ is a
subset of $R$.

As the posh set of $f$ is invariant under multiplication of $f$ by
non-zero constants and the rich set of $f$ only depends on the set of
roots $S$, we may assume that
\begin{equation*}
  f = \prod_{s\in S} (x-s)^{h_s} \in R[x]
\end{equation*}
where $S\subseteq R$ is finite and $h_s\in \N$ for $s\in S$. Let
$S'\subseteq S$ be a balanced set associated to $S$,
cf.~Definition~\ref{definition:assoc-balanced}.

Let $u$, $t\in S'$ with $u \neq t$ and $w \in u + M^{\rho_S(u)}$ and
$s \in t + M^{\rho_S(t)}$ be elements in two distinct partition blocks
of the  partition $\P_S = \P_{S'}$. Then the valuation $\val(w-s)$ only
depends on the blocks and not the specific choice of elements $w$ and
$s$, that is,
\begin{equation}\label{eq:val-diff-blocks}
  \val(w-s) = \val(u-t) \text{ for all } w\in u + M^{\rho_S(u)}, s\in t+M^{\rho_S(t)} \text{ with } u, t\in S', u\neq t.
\end{equation}
Therefore, if $w\in u + M^{\rho_S(u)}$ for $u\in S'$, then
\begin{align}\label{eq:eqs-p1}
  \begin{split}
    \val(f(w))
    &= \sum_{t\in S}h_t\val(w-t) = \sum_{t\in S'}\sum_{s\in S_t}h_s\val(w-s) \\
    &= \sum_{s\in S_u} \val(w-s)h_s + \sum_{\substack{t\in S' \\ t\neq u}}\val(u-t)\left(\sum_{s\in
      S_t}h_s\right).
  \end{split}
\end{align}
Observe that the second summand in the last line of
Equation~\eqref{eq:eqs-p1} only depends on the block
$u + M^{\rho_S(u)}$, not the specific choice of $w$.

For the first summand, however, it makes a significant difference
whether $w$ is element of an $S$-rich or an $S$-poor neighborhood.  If
$w$ is in an $S$-poor neighborhood of the block $u+M^{\rho_S(u)}$,
then
\begin{equation}
  \label{eq:poor-1}
  \val(w-s) = \val(w-u) = \rho_S(u) \text{ for all } s\in S_u.
\end{equation}
Now, if $r$ is in an $S$-rich neighborhood of the partition
block of $u + M^{\rho_S(u)}$, then the following hold:
\begin{itemize}
\item $\val(r-s) \ge \rho_S(u)$ for all $s \in S_u$ and
\item there exists $s\in S\cap r+M^{\rho_S(s)+1}$ such that
  $\val(r-s) > \rho_S(u)$.
\end{itemize}
Therefore, for this choice of $w$ and $r$, it follows that
\begin{equation*}
  \val(f(r)) - \val(f(w)) = \sum_{s\in S_u} (\val(r-s) - \rho_S(u))h_s >0
\end{equation*}
and hence
\begin{equation}
  \label{eq:rich>poor}
  \val(f(r)) > \val(f(w)) \ge \min_{t\in R}\val(f(t)) = \val(\fixdiv_f)
\end{equation}
where $\fixdiv_f$ is a generator of the fixed divisor of $f$.  This
immediately implies the following lemma

\begin{lemma}\label{lemma:RF_lemma-1}
  Let $(R,M)$ be as in Convention~\ref{conventionRD}, $S\subseteq R$ a
  finite set and $f = \prod_{s\in S}(x-s)^{m_s}$ with $m_s\in \N$ for
  $s\in S$.

  Then every element in an $S$-rich neighborhood of the partition
  $\P_S$ is in the posh set $\F f$, that is, $\R f \subseteq \F f$.

  In other words, every element $w\in R$ with
  $\val(f(w)) = \min_{t\in R}\val(f(t))$ is in an $S$-poor
  neighborhood of $\P_S$.
\end{lemma}

\section{Characterizing split polynomials with (relatively) small posh
  sets}\label{subsection:small-fat-sets}

There are two ways to characterize small posh sets. On one hand, the
inclusion $\R f \subseteq \F f$ (Lemma~\ref{lemma:RF_lemma-1}) means
that $\F f$ is smallest possible if $\F f = \R f$
(Lemma~\ref{FR_lemma}). On the other hand, we can measure the size of
the posh set by a finitely additive probability measure
(Lemma~\ref{RF_lemma}).

\begin{definition}\label{measuredef}
  For an ideal $I$ of finite index in $R$ let ${{|\!|}I{|\!|}}=[R:I]$
  and let $\sigma$ be the finitely additive probability measure on $R$
  defined by the requirement
  \[\sigma(r+I)= \frac{1}{{{|\!|}I{|\!|}}}\]
  whenever $r\in R$ and $I$ is an ideal of finite index in $R$;
  cf.~\cite{FriPaTiWi:1999:additive-measures}.  For our purposes, all
  we need to know about $\sigma$ is the values that it takes on finite
  unions of residue classes of ideals of finite index.
\end{definition}

\begin{remark}\label{remark:sigma-incl}
  Let $(R,M)$ be a local ring with finite residue field $R/M$ of
  order~$q$.
  \begin{enumerate}
  \item If $S\subseteq R$ is a finite set and $n_s\in \N$ for $s\in S$,
    then
    \begin{equation*}
      \sigma\left(\biguplus_{s\in S} s+M^{n_s} \right) = \sum_{s\in S} \frac{1}{q^{n_s}}.
    \end{equation*}
  \item If $A \subseteq B$ are subsets of $R$ that are each a finite
    union of residue classes of powers of $M$, then $A=B$ if and only
    if $\sigma(A) = \sigma(B)$.
\end{enumerate}
\end{remark}

Next, we discuss the question under which conditions the rich set and
the posh set have minimal $\sigma$-measure.

\begin{lemma}\label{RF_lemma}
  Let $(R,M)$ be as in Convention~\ref{conventionRD} and $|R/M|=q$.
  Let $f\in K[x]$ such that $f$ splits over $K$ whose root set $S$ is
  a subset of $R$.

  Then
  \begin{enumerate}
  \item\label{ReqF} $\R f = \F f$ if and only if
    $\sigma(\R f) = \sigma(\F f)$.
  \item\label{minsigma} $\sigma(\R f) = \sigma(\R S) \ge \frac{1}{q}$.
    Equality holds if and only if every block of $\P_S$ contains
    only one rich neighborhood.
  \item\label{Fminsigma} $\sigma(\F f)= \frac{1}{q}$ if and only if
    $\R f = \F f$ and every block of $\P_S$ contains only one rich
    neighborhood.
  \end{enumerate}
\end{lemma}
\begin{proof}
  Recall that $\R f \subseteq \F f$ by Lemma~\ref{lemma:RF_lemma-1}.

  Ad \eqref{ReqF}.  $\F f$ and $\R f$ are each a finite union of
  residue classes of various powers of $M$, and so is
  $\F f\setminus \R f$.  For sets of this kind, $\sigma$ takes a
  positive value unless they are empty.

  Ad \eqref{minsigma}.  Each block $C=z+M^n$ of the partition $\P_S$
  of $R$ contains at least one $S$-rich neighborhood $s + M^{n+1}$ and
  therefore
  $\sigma(\R f\cap C)\ge \frac{1}{q^{n+1}} = \frac{1}{q}\sigma(C)$.
  It follows that
  \[
    \sigma(\R f) = \sum_{C\in\P_S} \sigma(\R f\cap C)\ge
    \frac{1}{q}\sum_{C\in\P_S} \sigma(C) = \frac{1}{q}.
  \]
  By the same token, $\sigma(\R f)=\frac{1}{q}$ if and only if
  $\sigma(\R f\cap C)=\frac{1}{q} \sigma(C)$ for every block of the
  partition. Since every block contains at least one rich
  neighborhood, this is equivalent to saying that every block of the
  partition contains exactly one rich neighborhood.

  \eqref{Fminsigma} now follows from \eqref{ReqF} and
  \eqref{minsigma}.
\end{proof}

\begin{remark}\label{remark:balanced-sigma}
  Balanced sets have two properties:
  \begin{enumerate}
  \item The rich set $\R S$ of a balanced set has the minimal possible
    $\sigma$-measure. This is equivalent to each block of the
    associated partition having exactly one rich neighborhood,
    see~Lemma~\ref{RF_lemma}(2).
  \item Every balanced set is minimal with respect to inclusion among
    all finite sets sharing the same rich set. This property is
    equivalent to every rich neighborhood containing only one element
    of the underlying set.
  \end{enumerate}
  Note that balanced sets are characterized by these two properties. 
\end{remark}

We now characterize the case where $\R f = \F f$ holds in terms of
root multiplicities. For this purpose, we revisit the observations
made before Lemma~\ref{lemma:RF_lemma-1} and recall some notation.

Namely, suppose $S\subseteq R$ is a finite (not necessarily balanced)
set, $S'$ a balanced set associated to $S$ and for $t\in S'$ let
$S_t=\{s\in S \mid s+M^{\rho_S(s)} = t+M^{\rho_S(t)}\}$,
cf.~Definition~\ref{definition:assoc-balanced}. Further, let
$f=\prod_{s\in S}(x-s)^{h_s}$ where $h_s$ is a positive integer for
each $s\in S$. If $w\in R$ is an element of an $S$-poor neighborhood of
the partition block $u + M^{\rho_S(u)}$ of $\P_S$, then by
Equations~\eqref{eq:eqs-p1} and~\eqref{eq:poor-1}, it follows that
\begin{equation}\label{eq:poor-eqs}
  \val(f(w)) = \rho_S(u)\sum_{s\in S_u}h_s + \sum_{u\neq t\in S'}\val(u-t)\left(\sum_{s\in S_t}h_s\right).
\end{equation}
Note that the right hand side of this equality does not depend on the
specific choice of $w$, that is, $\val(f(w))$ is the same value for
all $w\in u+M^{\rho_S(u)} \setminus \R f$.

For $t\in S'$, let $m_t = \sum_{s\in S_t}h_s$ denote the number of
roots $s$ of $f$ (counting multiplicities) which are elements of the
block $t + M^{\rho_{S}(s)}$. Then, by Equation~\eqref{eq:poor-eqs},
the column vector $(m_s)_{s\in S'}$ is a solution over the positive
integers to the linear equation system
\begin{equation}\label{eq:eqs-p2}
  \sum_{\substack{t\in S' \\ t\neq u}}\val(t-u)x_t + \rho_S(u)x_u
  = \val(f(w)) \text{ with } u \in S',
\end{equation}
where $w\in u+M^{\rho_S(u)} \setminus \R f$ is fixed, but arbitrary.
This motivates the following definition.

\begin{definition}\label{matrix_def}
  Let $\P$ be an $M$-adic partition of $R$ and $T$ a set of
  representatives of $\P$, that is,
  \begin{equation*}
    \P = \P_{T} = \{s + M^{\rho_{T}(s)} \mid s\in T\}
  \end{equation*}
  as in Definition~\ref{def:Mpart-bal}.  We define the \emph{partition
    matrix} of $\P$ to be the square matrix $A=A_{\P}$ whose rows and
  columns are indexed by $\P$, or, equivalently, by $T$:
  $A_\P=(a_{s,t})_{s,t\in T}$, where
  \[
    a_{s,t} =
    \begin{cases}
      \rho_{T}(s)  & s = t \\
      \val(s-t)      & s\neq t
    \end{cases}.
  \]
  Note that $A_{\P}$ depends only on $\P$, not on the system $T$ of
  representatives.
\end{definition}

Recall that $\R f = \F f$ holds if and only if for each $w\in R$
which is an element of an $S$-poor neighborhood
\begin{equation*}
  \val(f(w)) = \min_{t\in R}\val(f(t)) 
\end{equation*}
holds. Considering that $\min_{t\in R}\val(f(t)) = \val(\fixdiv_f)$,
this means that the right hand side of the linear equation
system~\eqref{eq:eqs-p2} is the vector
$\left(\val(\fixdiv_f)\right)_{s\in S'}$, whose every coordinate is
the same. We introduce the following notation to make it easier to
refer to a vector of this form.

\begin{notation}
  For $n\in \N$, let $\one_n = (1)_{1\le i \le n}$ denote the vector
  whose every entry equals 1. Moreover, we write
  $\x = (x_i)_{1\le i \le n}$ for a vector of indeterminates where we
  omit the size $n$ for better readability.
\end{notation}

\begin{lemma}\label{FR_lemma}
  Let $(R,M)$ be as in Convention~\ref{conventionRD} and
  \begin{equation*}
    f = \prod_{s\in S} (x-s)^{h_s}
  \end{equation*}
  for a finite subset $S \subseteq R$ and positive integers $h_s$ for
  $s\in S$.

 Further, let $A = A_{\P_S}$ be the partition matrix of $\P_S$,
 $S' \subseteq S$ be a balanced set associated to $S$ and
 $n = |S'| = |\P_S|$.

 If $m_t = \sum_{s\in S_t}h_s$ denotes the number of roots of $f$
 (counted with multiplicities) in the partition block
 $t + M^{\rho_S(t)}$ for $t\in S'$
 (cf.~Definition~\ref{definition:assoc-balanced}), then the following
 assertions are equivalent:
 \begin{enumerate}
 \item $\F f = \R f$.
 \item For every $r$ in an $S$-poor neighborhood,
   $\val(f(r))=\val(\fixdiv_f)$.
 \item For every $r$ in an $S$-poor neighborhood, $\val(f(r))$ is the
   same value.
 \item The column vector $(m_t)_{t\in S'}$ is a solution to
   $A\x = \val(\fixdiv_f)\one_{n}$.
 \item The column vector $(m_t)_{t\in S'}$ is a solution to
   $A\x = e\one_{n}$ for some $e\in \N_0$ where $e = 0$ if
   and only if $n = 1$.
 \end{enumerate}
\end{lemma}

\begin{proof}
  Recall $\F f = \{s\in R\mid \val(f(s)) > \val(\fixdiv_f)\}$ and that
  $\R f \subseteq \F f$ holds by
  Lemma~\ref{lemma:RF_lemma-1}(1). Moreover,
  $\val(\fixdiv_f) = \val(f(w))$ for some $w\in R$ which is contained
  in an $S$-poor neighborhood. Also, $|S'|>1$ if and only if $\P_S$
  contains more than one partition block. In this case $S$ contains a
  complete set of residues modulo $M$ which is equivalent to
  $\val(\fixdiv_f) > 0$.

  Now, $\F f = \R f$ if and only if no $S$-poor neighborhood is
  contained in the posh set $\F f$, meaning, for every $w$ in an
  $S$-poor neighborhood, $\val(f(w))=\val(\fixdiv_f)$.  This is the
  case if and only if $\val(f(w))$ is the same for all elements $w$ of
  $S$-poor neighborhoods.  Since by Equation~\eqref{eq:eqs-p2}
  \[
    \sum_{t\in S'} a_{u, t}m_t = \val(f(w))
  \]
  holds for every $w$ in an $S$-poor neighborhood of the block
  $u+M^{\rho_S(u)}$ with $u\in S'$, the result follows.
\end{proof}

Given a split polynomial $f$, Lemma~\ref{FR_lemma} provides us with
an easy computational method to check whether $\R f = \F f$
holds. In addition, we can also use the equivalent assertion of the lemma
to construct an absolutely irreducible polynomial whose root set is a
given balanced set.

\section{Constructing split absolutely irreducible polynomials}
\label{sec:equalizing-polynomial}

The goal of this section is to prove
Theorem~\ref{theorem:equpoly-exists}.  Given a balanced set $S$ we
construct a (uniquely determined) absolutely irreducible polynomial
whose root set is $S$, by computing root multiplicities
$(m_s)_{s\in S}$ such that $f = \prod_{s\in S}(x-s)^{m_s}$ satisfies a
condition which by Proposition~\ref{proposition:special-abs-irr}
implies that $\fixdiv_f^{-1}f$ is absolutely irreducible.

For this we show that for a balanced set $S$ there exists a vector of
multiplicities such that the corresponding polynomial satisfies
condition~(3) of Lemma~\ref{FR_lemma}. We then prove that the system
matrix of the linear equation system in condition~(5) of the same
lemma is non-singular. In the last part of the section we show that
the solution over the positive integers with minimal possible
right-hand side $e$ is the right choice for the multiplicities of the
roots.

\subsection{Choosing the multiplicities of the roots}

\begin{proposition}\label{Bpositive_lemma}
  Let $(R,M)$ be as in Convention~\ref{conventionRD}, $S$ a balanced
  set and $\P_S$ the partition associated to it
  (Definition~\ref{def:assoc-part-rich-poor}).

  Then there exists a vector $(m_s)_{s\in S}\in \N^{|S|}$ of positive
  integers such that $f=\prod_{s\in S} (x-s)^{m_s}$ satisfies: for
  every $r$ in an $S$-poor neighborhood of $\P_S$, $\val(f(r))$ takes
  the same value $e$.
\end{proposition}

\begin{proof}
  Let $m$ be maximal such that some residue class of $M^m$ is a block
  of $\P_S$.  By (reverse) induction from $n=m$ down to $n=0$ we show:

  For every residue class $C$ of $M^n$ that is a union of blocks of
  $\P_S$ we can find multiplicities $k_s$ for all $s\in S\cap C$ such
  that $f_C=\prod_{s\in S\cap C} (x-s)^{k_s}$ satisfies: for every $r$
  in a poor neighborhood of $\P_S$ contained in $C$, $\val(f_C(r))$
  takes the same value. The statement for $n=0$ proves the lemma,
  since then $C=R$.

  For $n=m$, $C\cap S$ contains a single element $s$, and ${k_s}=1$
  works.  Now let $C$ be a residue class of $M^n$ that is a union of
  blocks of $\P_S$. Either $C$ itself is a block of $\P_S$ (and we can
  set $k_s=1$ for the single element $s$ in $S\cap C$), or $C$ is the
  disjoint union of $C_i$ for $1\le i\le q$, each $C_i$ a residue
  class of $M^{n+1}$ that is a union of blocks of $\P_S$.

  In this case, we may assume, by induction hypothesis, that we have
  assigned a multiplicity $k_s$ to each $s\in S\cap C_i$ such that for
  each $i$, the polynomial
  \begin{equation*}
    f_i=\prod_{s\in S\cap C_i} (x-s)^{k_s}
  \end{equation*}
  satisfies: for every $r$ in a poor neighborhood of $\P_S$ contained in
  $C_i$, $\val(f_i(r))$ takes the same value, say, $a_i$.

  Further, note that, for each $t$ in any $C_j$ with $j\ne i$,
  $\val(f_i(t))$ takes the same value, say $e_i$, with $e_i<a_i$.  We
  set $d_i=a_i - e_i$ for $1\le i\le q$.

  Now let $c=\lcm(d_i\mid 1\le i\le q)$ and $c_i=\frac{c}{d_i}$.  For
  $s\in S\cap C_i$, set $h_s=c_i k_s $.  Then
  $f_C=\prod_{s\in S\cap C} (x-s)^{h_s} = \prod_{i=1}^q {f_i}^{c_i}$
  satisfies that $\val(f_C(r))$ takes the same value for all $r$ in any
  poor neighborhood of $\P_S$ contained in $C$.

  Indeed, if $r$ is in a poor neighborhood of $\P_S$ contained in
  $C_i$, then
  $v(f_C(r)) = \sum_{j=1}^q c_j v(f_j(r)) = c_i a_i + \sum_{j\ne i}
  c_j e_j$, where $c_i a_i = c_i (d_i + e_i) = c + c_i e_i$ and hence
  $v(f_C(r)) = c + \sum_{j=1}^q c_j e_j$, which does not depend on
  $i$.
\end{proof}

\subsection{The partition matrix is non-singular}

We assume here that $S$ is a balanced set with $|S|>1$, or
equivalently, the partition $\P_S$ contains more than one
block. (Otherwise, the partition matrix is $(0)$.  We treat this case
separately in the proof of Theorem~\ref{theorem:equpoly-exists} and
Proposition~\ref{proposition:unique-unimodular-solution}.)

\begin{lemma}\label{lemma:special-det}
Let $u_1$, $u_2$, \ldots, $u_n$, $b\in \QQ_{>0}$ be positive
  rationals such that $u_i > b$ for all $1\le i \le n$ and let
  \begin{equation*}
    A =
    \begin{pmatrix}
      u_1     & b   & \cdots  & b \\
      b       & u_2 & \cdots  & b \\
      \vdots  &     & \ddots  & \vdots  \\
      b      & b   & \cdots  & u_n \\
    \end{pmatrix}.
  \end{equation*}

  Then $\det(A) > 0$.
\end{lemma}
\begin{proof}
  We prove by induction on $n$. If $n=1$, then $\det(A) = u_1 >0$
  which proves the basis.

  For $n\ge 2$, we eliminate the off-diagonal entries in the first
  column by adding the $\frac{-b}{u_1}$-fold of the first row to all
  other rows. The resulting matrix is
  \begin{equation*}
    \begin{pmatrix}
      u_1     & b  \,\cdots\,   b \\
      0       & \hat A
    \end{pmatrix}
  \end{equation*}
  where
  \begin{equation*}
    \hat A =
    \begin{pmatrix}
     u_2-\frac{b^2}{u_1} & \cdots  & b-\frac{b^2}{u_1} \\
         & \ddots  & \vdots  \\
     b-\frac{b^2}{u_1}   & \cdots  & u_n-\frac{b^2}{u_1} \\
    \end{pmatrix}.
  \end{equation*}
  Since $u_i > b >0$ for all $1\le i \le n$, it follows that
  \begin{equation*}
    u_i - \frac{b^2}{u_1} > b-\frac{b^2}{u_1} > 0.
  \end{equation*}
  Therefore $\hat A$ satisfies the assumptions of the lemma and by
  induction hypothesis it follows that $\det(\hat A) >0$ and
  \begin{equation*}
    \det(A) = u_1\det(\hat A) >0.
  \end{equation*}
\end{proof}

\begin{notation} For $n$, $t\in \N$, we write
  $\bfE_{n,t} = (1)_{\substack{1\le i \le n\\ 1 \le j \le t}}$ for the
  $(n\times t)$-matrix all of whose entries are 1. Note that
  $\one_n = \bfE_{n,1}$.
\end{notation}

\begin{lemma}\label{lemma:liftblock}
  Let $A\in \QQ^{n\times n}$ such that the equation system
  $A\vec{x} = \one_n$ has a solution over the positive rationals.

  If $\det(A) \neq 0$, then $\det(A+\bfE_{n,n}) \neq 0$.
\end{lemma}
\begin{proof}
  Let $\vec{u} = (u_i)_{1\le i \le n} \in \QQ_{>0}^n$ with
  $A\vec{u} = \one_n$ and let $\vec{a}_j$ denote the $j$-th column of
  $A$ for $1\le j\le n$, that is, $\sum_{j=1}^nu_j\vec{a}_j = \one_n$.

  Let $\lambda_1$, \ldots, $\lambda_n\in \QQ$ such that
  \begin{align}
    \label{eq:lincomb}
    0
    &= \sum_{j=1}^n\lambda_j\left(\vec{a}_j + \one_n\right)
    = \sum_{j=1}^n\lambda_j\vec{a}_j + \left(\sum_{i=1}^n\lambda_i\right)\one_n \\
    &= \sum_{j=1}^n\lambda_j\vec{a}_j + \left(\sum_{i=1}^n\lambda_i\right)\left(\sum_{j=1}^nu_j\vec{a}_j\right) \\
    &= \sum_{j=1}^n\left(\lambda_j + u_j\sum_{i=1}^n\lambda_i\right)\vec{a}_j
  \end{align}
  Since $\det(A) \neq 0$, it follows that
  \begin{equation*}
    \lambda_j + u_j\sum_{i=1}^n\lambda_i
    = \lambda_j(1+u_j) + u_j\sum_{\substack{i=1 \\ i \neq j}}^n\lambda_i= 0
  \end{equation*}
  for all $1\le j \le n$. In other words,
  $(\lambda_1, \ldots, \lambda_n)$ is a solution of the homogeneous
  linear equation system $U\vec{x} = 0$ where
  \begin{equation*}
    U =
    \begin{pmatrix}
       u_1 +1 & u_1    & \cdots & u_1  \\
       u_2    & u_2 +1 & \cdots & u_2  \\
       \vdots &        & \ddots & \vdots  \\
       u_n    & u_n    & \cdots & u_n+1  \\
     \end{pmatrix}.
   \end{equation*}
   Since $u_i > 0$ for $1\le i \le n$ and
   \begin{equation*}
     \det(U)
     = u_1u_2\cdots u_n \cdot \det
    \begin{pmatrix}
       1 + \frac{1}{u_1} & 1                 & \cdots & 1  \\
       1                 & 1 + \frac{1}{u_2} & \cdots & 1  \\
       \vdots            &                   & \ddots & \vdots  \\
       1                 & 1                 & \cdots & 1 + \frac{1}{u_n}  \\
     \end{pmatrix}
   \end{equation*}
   it follows from Lemma~\ref{lemma:special-det} that $\det(U) \neq
   0$. We conclude that $\lambda_1 = \cdots = \lambda_n = 0$ and
   therefore the columns of $A + \bfE_{n,n}$ are linearly independent.
\end{proof}


\begin{proposition}\label{proposition:detneq0}
  Let $(R,M)$ be as in Convention~\ref{conventionRD}. Further, let
  $\P$ be an $M$-adic partition of $R$ consisting of more than one
  block and $A_\P$ the partition matrix of $\P$.

  Then $\det(A_\P) \neq 0$.
\end{proposition}
\begin{proof}
  Let $S$ denote a set of representatives of the blocks of $\P$. Then
  $S$ is balanced, $|S|>1$ by assumption, $\P = \P_S$ and the $S$ can
  be used as index set for the rows and columns of $A$.

  Let $n\in \N$ be maximal such that $s+M^n$ is a partition block of
  $\P$. We prove the assertion by induction on $n$.  If $n=1$, then
  $S$ is a set of representatives of the residue classes of $R$ modulo
  $M$, $a_{s,s} = 1$ and $a_{s,t} = 0$ for all $s\neq t\in S$. Hence
  $A$ is the identity matrix and $\det(A)>0$.

  Now let $n\ge 2$. As index set of the rows and columns we assume
  that $S$ is endowed with a fixed linear ordering of its elements in
  such a way that blocks contained in the same residue class of $M$
  are adjacent. This makes the matrix a block diagonal matrix with
  blocks $B_1$, \ldots, $B_q$ each block belonging to one residue
  class of $M$. Since $\det(A) = \det(B_1)\cdots \det(B_q)$, it
  suffices to show that the determinants of each block are positive.

  We fix $i$ and set $B = B_i$ and let $s$ be the size of $B$ and
  $\tilde S \subseteq S$ the subset of $S$ that serves as index set
  for $B$. If $|\tilde S| = 1$, then $B = (1)$ and hence
  $\det(B) \neq 0$. So assume that $|\tilde S|> 1$.

  We now construct a new partition $\tilde\calP$ of $R$ by
  (bijectively) mapping the elements of $\tilde S$ to a set of
  representatives $T$ of $\tilde\P$. We do this as follows.  Let
  $a_1$, \ldots, $a_q$ be a set of representatives of $R$ modulo $M$
  and let $t\in R$ with $M = tR$ be a uniformizer of $R$. Then for
  each $n\in \N$ and each residue class of $M^n$ there is a
  uniquely determined representative of the form
  \begin{equation*}
    \sum_{j=0}^{n-1} a_{i_j}t^j \text{ with }1\le i_0, \ldots, i_{n-1} \le q.
  \end{equation*}

  Without restriction, let $a_1$ be the representative modulo $M$ of
  all the elements in $\tilde S$.  We define the map
  \begin{align*}
    \ell\colon \bigcup_{n\ge 1} R/M^n &\to \bigcup_{n\ge 0} R/M^n \\
    \left(a_1 + \sum_{j=1}^{n-1}a_{i_j}t^j\right) + M^n &\mapsto \left(\sum_{j=1}^{n-1}a_{i_j}t^{j-1}\right) + M^{n-1}
  \end{align*}
  which bijectively maps the residue classes of $R$ modulo $M^n$ with
  $n>1$ contained in $a_1 + M$ to all residue classes of $R$ modulo
  $M^{n-1}$. For each $s\in \tilde S$, let $t(s)\in R$ such that
  \begin{equation*}
    \ell\!\left(s + M^{\rho_S(s)}\right) = t(s) + M^{\rho_S(s)-1}
  \end{equation*}
  and set
  \begin{equation*}
    T = \{t(s) \mid s\in \tilde S\}.
  \end{equation*}
  Since $\P_S$ is a partition of $R$, it follows that
  \begin{equation}
    a_1+M =  \bigcup_{s\in \tilde S} s + M^{\rho_S(s)}
  \end{equation}
  and hence
  \begin{equation*}
    \tilde\calP
    = \left\{ \ell\!\left(s + M^{\rho_S(s)}\right) \mid s\in \tilde S \right\}
    = \left\{ t + M^{\rho_S(s)-1} \mid t\in T \right\}
  \end{equation*}
  is an $M$-adic partition of $R$. By construction, $T$ is a set of
  representatives of $\tilde\calP$ and $T$ is a balanced set with
  $|T|=|\tilde S|>1$. Moreover, $\rho_T(t(s)) = \rho_S(s)-1 = b_{s,s}-1$
  and
  $\val(t(s) - t(\tilde s)) = \val(s - \tilde s) -1 = b_{s, \tilde s}
  - 1$ for all $s$, $\tilde s\in \tilde S$.

  Therefore, $B - \bfE_{|T|,|T|}$ is the partition matrix of
  $\tilde \calP$ and since
  $\max_{t\in T} \rho_T(t) < \max_{s\in \tilde S}\rho_S(s)$ it follows
  by induction that $\det(B - \bfE_{|T|,|T|}) \neq 0$.

  Moreover, by Proposition~\ref{Bpositive_lemma}, there exists a
  positive solution vector to the system
  $(B - \bfE_{|T|,|T|})\vec{x} = \one_s$. Hence $\det(B) \neq 0$
  according to Lemma~\ref{lemma:liftblock} which completes the proof.
\end{proof}

\subsection{The equalizing polynomial}

\begin{definition}
  An integer vector $(m_i)_{1\le i \le n} \in \Z^n$ is called
  \emph{unimodular} if
  \begin{equation*}
    \gcd(m_i\mid 1\le i\le n) = 1.
  \end{equation*}
\end{definition}

\begin{proposition}\label{proposition:unique-unimodular-solution}
  Let $(R,M)$ be as in Convention~\ref{conventionRD}, $S$ be a
  balanced subset of $R$, $n=|S|$ and $A_{\P_S}$ the partition matrix
  of the partition $\P_S$ associated to $S$.

  Then there exists a uniquely determined unimodular solution
  $(m_s)_{s\in S}\in \N^{n}$ of $A_{\P_S}\vec{x} = e\one_n$ with
  $e\in \N_0$.

  In addition, if $(k_s)\in \N^{n}$ is a solution to
  $A_{\P_S}\vec{x} = \tilde e\one_n$ with $\tilde e\in \N_0$, then
  $k_s = \ell m_s$ for all $s\in S$ and $\tilde e = \ell e$ for some
  $\ell \in \N$.
\end{proposition}
\begin{proof}
  By Proposition~\ref{Bpositive_lemma}, there exist $(m_s)_{s\in S}$
  such that $f = \prod_{s\in S}(x-s)^{m_s}$ is a polynomial for
  which $\val(f(w)) = e \in \N_0$ is the same for each $w\in R \setminus \R
  f$. Therefore, by Lemma~\ref{FR_lemma}, $(m_s)_{s\in S}$ is a
  solution of the linear equation system $A_{\P_S} \vec{x} = e\one_n$.
  Common integer factors of the coordinates of $(m_s)_{s\in S}$
  necessarily divide $e$.  By cancelling them out we can assume that
  $(m_s)_{s\in S}$ is a unimodular vector (and $e\in \N_0$).

  If $n=1$, then $A_{\P_S} = (0)$ which implies that $e=0$ and $(1)$
  is the uniquely determined unimodular integer solution to the
  equation system. The second assertion of the proposition immediately
  follows.
  
  From now on we assume that $n>1$ and hence $e>0$ (see
  Lemma~\ref{FR_lemma}(5)). Further, let $(k_s)\in \N^{n}$ is a
  solution vector to $A_{\P_S}\vec{x} = \tilde e\one_n$ with
  $\tilde e\in \N_0$.
  
  Let $c = \gcd(e, \tilde e)$ and write
  $e = cd$ and $\tilde e = c \ell$ for suitable (coprime) positive
  integers $d$ and $\ell$. Then
  \begin{equation*}
    A_{P_S}\left(\frac{m_s}{d}\right)_{s\in S}
    = c\one_{n}
    =  A_{P_S}\left(\frac{k_s}{\ell}\right)_{s\in S}.
  \end{equation*}

  By assumption $n>1$ and matrix $A_{\P}$ has non-zero determinant by
  Proposition~\ref{proposition:detneq0}. Therefore
  $\left(\frac{m_s}{d}\right)_{s\in S} =
  \left(\frac{k_s}{\ell}\right)_{s\in S}$. In other words,
  \begin{equation*}
    \ell m_s = dk_s \text{ for all } s\in S
  \end{equation*}
  and since $d$ and $\ell$ are coprime, it follows that $d \mid m_s$
  for all $s\in S$. As $(m_s)_{s\in S}$ is a unimodular vector, it
  follows that $d=1$, $\ell m_s = k_s$ for all $s\in S$ and
  $\tilde e = \ell e$.
\end{proof}

\begin{definition}\label{definition:equ-poly}
  Let $(R,M)$ be as in Convention~\ref{conventionRD} and let
  $S \subseteq R$ be a balanced set, $n = |S|$ and $A$ the partition
  matrix of the partition associated to $S$
  (Definition~\ref{matrix_def}). We define the \emph{equalizing
    polynomial of $S$} as
  \begin{equation*}
    f= \prod_{s\in S}(x-s)^{m_s} 
  \end{equation*}
  where $(m_s)_{s\in S}\in \N^{n}$ is the uniquely determined
  unimodular solution over the positive integers to
  $A\vec{x} = e\one_{n}$ with $e\in \N_0$.
\end{definition}

\begin{remark}\label{rem:csr-equpoly}
  In the special case where $S$ is a complete set of residues of
  $M^k$, all roots of the equalizing polynomials are simple roots. The
  vector of multiplicities~$\one_n$ is a solution of $A\x = e\one_n$
  for some $e\in \N$ because every row contains the same elements in a
  different order. It is therefore the unique unimodular solution.
\end{remark}

 \begin{lemma}\label{lemma:equpoly-rich-posh}
   Let $(R,M)$ be as in Convention~\ref{conventionRD}, $S\subseteq R$
   a balanced subset, $f$ the equalizing polynomial of $S$.

  Then $\R {f} = \F {f}$.
\end{lemma}
\begin{proof}
  This follows from the definition of the equalizing polynomial and
  Lemma~\ref{FR_lemma}.
\end{proof}

By means of the equalizing polynomial, we are now ready to show that
every balanced set occurs as the root set of a split absolutely irreducible
polynomial.

\begin{theorem}\label{theorem:equpoly-exists}
  Let $R$ be a discrete valuation domain with finite residue field,
  $S\subseteq R$ a balanced subset, $f$ the equalizing polynomial of
  $S$ and $\fixdiv_{f}$ a generator of the fixed divisor of $f$.

  Then the essentially unique image-primitive polynomial
  $F\in \Int(R)$ associated in $K[x]$ to~$f$, namely
  $F = \fixdiv_{f}^{-1}f$, is absolutely irreducible.
\end{theorem}

\begin{proof}
  $\R {f} = \F {f}$ holds by Lemma~\ref{lemma:equpoly-rich-posh} and
  $\sigma(\F {f}) =\frac{1}{q}$ by Lemma~\ref{RF_lemma}.

  Let $n=|S|$. First assume that $n=1$. Then $f = x-s$ and hence
  absolutely irreducible.

  From now on, assume $n>1$. We use
  Proposition~\ref{proposition:special-abs-irr} to show that $F$ is
  absolutely irreducible. Let $g\in K[x]$ be monic with $g\neq 1$
  satisfying $\Q {g}\subseteq \Q{f}$. This is equivalent to
  $g = \prod_{t\in T}(x-t)^{k_t}$ for $\emptyset \neq T\subseteq S$
  and $k_t\in \N$ for $t\in T$ (cf.~Remark~\ref{remark:split-and-Q}).
  Also assume that $\F {g}\subseteq \F{f}$; we show that $g$ is a
  power of $f$.
  
  Since $\F g \subseteq \F f$ and
  $\sigma(\F{g})\ge\frac{1}{q} = \sigma(\F {f})$ holds by
  Lemma~\ref{RF_lemma}, it follows that $\F g = \F f$ (see
  Remark~\ref{remark:sigma-incl}(2)).
  Using Lemma~\ref{RF_lemma}, we conclude that
  \begin{equation*}
    \R{g} = \F{g} = \F{f} = \R{f}.
  \end{equation*}
  On one hand, this implies that $S = T$ since no  proper subset
  of balanced set can have the same rich set, see
  Remark~\ref{remark:balanced-sigma}(2).

  On the other hand, by Lemma~\ref{FR_lemma}, $\R {g} = \F {g}$
  implies that $(k_s)_{s\in S}$ is a solution to
  $A\vec{x} = e\one_{n}$ for some positive integer
  $e \in \N$ where $A$ denotes the partition matrix of the
  partition associated to $S$. By
  Proposition~\ref{proposition:unique-unimodular-solution} it follows
  that $k_s = \ell m_s$ with $\ell\in \N$ for all $s\in S$, that is,
  $g = f^\ell$. 
\end{proof}


Theorem~\ref{theorem:equpoly-exists} also yields sufficient conditions
for absolute irreducibility in the global case.

\begin{corollary}\label{corollary:globalizing}
  Let $D$ be a Dedekind domain, $f = \prod_{s\in S}(x-s)^{m_s}$ for a
  finite subset $S$ of $D$ and $m_s\in \N$ for $s\in S$ such that the
  fixed divisor of $f$ is a principal ideal, generated by
  $\fixdiv_{f}\in D$.

  If there exists a prime ideal $P$ of $D$ with finite residue field
  such that $S$ is $(D_P,P_P)$-balanced and $f$ is the equalizing
  polynomial of $S$, then $\fixdiv_{f}^{-1}f$ is absolutely
  irreducible in $\Int(D)$.
\end{corollary}

\begin{proof}
  $\Int(D)$ behaves well under localization, that is,
  $\Int(D)_P = \Int(D_P)$ holds for every Noetherian domain $D$,
  cf.~\cite[Theorem~I.2.3]{Cahen-Chabert:1997:book}. Therefore
  $F = \fixdiv_{f}^{-1}f \in \Int(D)$ is an element of $\Int(D_P)$ and
  absolutely irreducible in $\Int(D_P)$ by
  Theorem~\ref{theorem:equpoly-exists}.

  Now, let $m\in \N$ and assume that $F^m = G_1G_2$ for non-constant
  polynomials $G_1$ and $G_2$ in $\Int(D)$.  It follows from the
  absolute irreducibility of $f$ in $\Int(D_{P})$ that, for $i=1,2$,
  there exist integers $e_i\ge 1$ with $e_1 + e_2 = m$ and non-zero
  elements $u_i$ in the quotient field of $D$ with $u_1u_2 = 1$ such
  that
  \begin{equation*}
    G_i = u_iF^{e_i}.
  \end{equation*}
  Now, since $F^{e_i}$ is image-primitive and $G_i \in \Int(D)$, it
  follows that $u_i \in D$ for $i=1,2$. Since $u_1u_2=1$, $G_i$ is
  associated to $F^{e_i}$ in $\Int(D)$ for $i=1,2$. The assertion
  follows.
\end{proof}

\section{Characterization of split absolutely irreducible polynomials}
\label{sec:char}
We now give a completely general characterization of absolutely
irreducible polynomials in $\IntR$ which split over $K$. First, we
cover those whose roots are in~$R$; and, finally, all split absolutely
irreducible integer-valued polynomials are characterized in
Corollary~\ref{corollary:root-set-K-R}.

\begin{theorem}\label{characterization_thm}
  Let $R$ be a discrete valuation domain with finite residue field and
  $K$ its quotient field.  Let
  \[
    f = \prod_{s\in S} (x-s)^{m_s} \quad{\text{and}}\quad F = c^{-1}
    f,
  \]
  where $\emptyset\ne S\subseteq R$ is a finite set and for each
  $s\in S$, $m_s \in \N$ a positive integer and
  $c\in K\setminus\{0\}$.

  Then $F$ is absolutely irreducible in $\IntR$ if and only if
  \begin{enumerate}
  \item $S$ is balanced.
  \item $f$ is the equalizing polynomial of $S$.
  \item $c$ is a generator of the fixed divisor of $f$.
  \end{enumerate}
\end{theorem}
\begin{proof}
  The equivalence is trivially true when $|S|=1$. Now assume $|S|>1$.

  If $S$ is a balanced set, $f = f_S$ its equalizing polynomial and
  $c$ a generator of the fixed divisor of $f_S$, then $F$ is
  absolutely irreducible by Theorem~\ref{theorem:equpoly-exists}.

  Conversely, assume that $F$ is absolutely irreducible.  Let $\P_S$ be
  the partition associated to $S$.  Let $S'\subseteq S$ be a balanced
  set associated to $S$,
  cf.~Definition~\ref{definition:assoc-balanced}. Then $S'$ is a
  balanced set such that $\P_{S'} = \P_S$ and
  $\R {S'} \subseteq \R S$, cf.~Remark~\ref{remark:rho}(4).

  First, we show that $S = S'$. Let $g$ be the equalizing polynomial
  of $S'$. Then, using Lemmas~\ref{lemma:equpoly-rich-posh}
  and~\ref{lemma:RF_lemma-1} we conclude that
  \begin{equation}\label{eq:FR-g0-f0}
    \F{g} =  \R{g} = \R{S'} \subseteq \R{S} = \R{f} \subseteq \F{f}
  \end{equation}
  holds. Moreover, $\Q {g} \subseteq \Q {f}$ since
  $S' \subseteq S$. By Corollary~\ref{strict}, $F$ being absolutely
  irreducible implies that $\Q {g} = \Q {f}$ and $\F {g} = \F {f}$.

  Therefore $S = S'$ is a balanced set and $g = f_S$ its equalizing
  polynomial. It follows from Theorem~\ref{theorem:equpoly-exists}
  that $G = \fixdiv_{g}^{-1}g$ is absolutely irreducible and $f$ is a
  polynomial with $\Q{f} = \Q{g}$ and $\F{f} = \F{g}$. We conclude by
  Proposition~\ref{proposition:special-abs-irr} that $f = g^n$ and,
  via Lemma~\ref{power_lemma}, $F \assoc G^n$ for some $n\in
  \N$. Since $F$ is absolutely irreducible it follows that $n=1$ and
  $f = g$ is the equalizing polynomial of $S$ and $c$ a generator of
  its fixed divisor.
\end{proof}

The combination of Theorem~\ref{characterization_thm} and
Theorem~\ref{theorem:equpoly-exists} establishes a bijection between
those absolutely irreducible polynomials in $\IntR$ that split over
$R$ and balanced subsets of $R$.

\begin{corollary}\label{corollary:bijectioncor}
  Let $(R,M)$ be as in Convention~\ref{conventionRD}. We identify
  polynomials in $\IntR$ that differ only by multiplication by units
  of $R$.

  The absolutely irreducible polynomials of $\IntR$ of the form
  \[
    f=c^{-1}\prod_{s\in S} (x-s)^{m_s}
  \]
  with $\emptyset\ne S\subseteq R$, each $m_s$ a positive integer, and
  $c\in K\setminus\{0\}$ correspond bijectively to the balanced sets
  $S\subseteq R$.

  The bijective correspondence is as follows: given an absolutely
  irreducible polynomial $f$, map $f$ to its set of roots $S$.
  Conversely, given a balanced finite set $S\subseteq R$, let $f$ be
  its equalizing polynomial and $\fixdiv_{f}\in R$ a generator of
  the fixed divisor of $f$, and map $S$ to
  $F = \fixdiv_{f}^{-1}f$, that is, to the essentially unique
  image-primitive polynomial associated in $K[x]$ to the equalizing
  polynomial of $S$.
\end{corollary}

\begin{corollary}\label{corollary:root-set-K-R}
  Let $(R,M)$ be as in Convention~\ref{conventionRD}. The absolutely
  irreducible polynomials in $\IntR$ that split over $K$ are
  \begin{enumerate}
  \item those whose roots are in $R$ as described in
    Corollary~\ref{corollary:bijectioncor} and
  \item linear polynomials $ax-b$ with $a\in M$ and
    $b\in R \setminus M$.
  \end{enumerate}
\end{corollary}
\begin{proof}
  Every $F\in \IntR$ is of the form $F = \frac{cf}{d}$ where
  $f\in R[x]$ is a primitive polynomial and $c$, $d\in R$ with
  $\gcd(c,d)=1$. Moreover, if $F$ splits over $K$, then $f$ is a
  product of linear factors of the form $ax-b$ for coprime elements
  $a$ and $b$ in $R$. Therefore, exactly one of $a$ and $b$ is in the
  maximal ideal $M$ while the other is a unit of $R$. If $a$ is a
  unit, then $a^{-1}b\in R$. If, however, $a\in M$, then $b$ is a unit
  and $ar-b$ is a unit for all $r \in R$. So, $ax-b$ by itself is
  absolutely irreducible and cannot be a factor of an
  irreducible integer-valued polynomial of degree strictly greater
  than $1$.

  It follows that the only split absolutely irreducible polynomials in
  $\IntR$ that are not covered by the bijection in
  Corollary~\ref{corollary:bijectioncor} are the linear polynomials of
  the form $ax-b$ where $a\in M$ and $b\in R\setminus M$. 
\end{proof}

\begin{corollary}\label{corollary:rootsetsize}
  Let $(R,M)$ be as in Convention~\ref{conventionRD} with $|R/M| = q$
  and $F\in \Int(R)$ be a split polynomial with root set
  $S\subseteq R$.

  If $F$ is absolutely irreducible, then $|S| \equiv 1 \mod q-1$.
\end{corollary}
\begin{proof}
  Note that the cardinality of a balanced set equals the number of
  blocks of the associated $M$-adic partition. It is apparent from the
  construction of $M$-adic partitions in Lemma~\ref{lemma:S-partition}
  that we can obtain any $M$-adic partition by starting with the
  single block $R$ and then repeatedly replacing a block that is a
  residue class of $M^n$ by the $q$ residue classes of $M^{n+1}$
  contained in it, which increases the number of blocks by $q-1$.
\end{proof}


\section{Application to generalized binomial polynomials}
\label{sec:appl-gener-binom}
In this section we discuss the absolute irreducibility of the
integer-valued polynomials whose root sets are initial sequences of
$P$-orderings. Such sequences were already considered by
Pólya~\cite{Polya:1919:regb} and Ostrowski~\cite{Ostrowski:1919:regb}
in their investigation of regular bases of rings of integer-valued
polynomials on rings of integers in number fields. In the literature,
they are also known as \emph{very well distributed and very well
  ordered} sequences, see~\cite[Ch.~2]{Cahen-Chabert:1997:book}. We
follow here the terminology introduced by
Bhargava~\cite[Section~2]{BM1997:P-orderings}.

\begin{definition}
  \label{def:p-orderings}
  Let $D$ be a Dedekind domain and $P$ a maximal ideal of $D$ and
  $\val_P$ the discrete valuation associated to $P$.  A
  \emph{$P$-ordering of $D$} is a sequence $(a_i)_{i\ge 0}$ in $D$ which
  satisfies
  \begin{equation*}
    \val_P\!\left(\prod_{i=0}^{k-1}(a_k - a_i)\right)
    \le \val_P\!\left(\prod_{i=0}^{k-1}(a - a_i)\right)
  \end{equation*}
  for all $k\ge 0$ and all $a\in D$.
\end{definition}

\begin{remark}
  The sequence $(i)_{i\ge 1}$ of consecutive natural numbers is a
  $p\Z$-ordering of the ring $\Z$ of integers for each prime number
  $p$.
\end{remark}
\begin{fact}[{Bhargava~\cite[Theorem~1, Lemma~3]{BM1997:P-orderings}}]
  The sequence of $P$-adic valuations
  \begin{equation*}
    \alpha_P(k) \coloneqq \val_P\!\left(\prod_{i=0}^{k-1}(a_k - a_i)\right)
  \end{equation*}
  does not depend on the choice of the $P$-ordering; it is intrinsic
  to $D$.  Moreover, for each $k\ge 0$, there are only finitely many
  prime ideals for which $\alpha_P(k) \neq 0$.
\end{fact}

\begin{definition}
  \label{def:generalized-factorial}
  Let $D$ be a Dedekind domain and $k\in \N_0$. The \emph{generalized
    factorial of $k$} with respect to $D$ is defined as the ideal
  \begin{equation*}
    \genfac_D(k) = \prod_{P \in \max(D)} P^{\alpha_P(k)}
  \end{equation*}
  where $\max(D)$ denotes set of maximal ideals of $D$.
\end{definition}

\begin{remark}
  For $n\in \N$, the generalized factorial $\genfac_{\Z}(n) = n!\Z$ is
  the ideal generated by the usual factorial $n!$.
\end{remark}

\begin{remark}\label{remark:p-ordering-properties}
  Let $D$ be a Dedekind domain and $P$ a maximal ideal of $D$ whose
  residue field is of finite order $q$ and let $(a_i)_{i\ge 0}$ be a
  $P$-ordering and $n\in\N$. 
 \begin{enumerate}
  \item\label{item:valuations} Every choice of $q^n$ consecutive
    elements of $(a_i)_{i\ge 0}$ is a complete system of residues
    modulo $P^n$. Every complete system of residues modulo $P^n$ is
    the initial sequence of a $P$-ordering.
  \item\label{item:localize-p-orderings} The sequence $(a_i)_{i\ge 0}$
    is a $PD_P$-ordering of the localization $D_P$ of $D$ at $P$ and a
    $P\widehat D$-ordering of the $P$-adic completion $\widehat D$
    of $D$.
  \end{enumerate}
\end{remark}

\begin{remark}\label{remark:fixdiv}
  Note that for any $P$-ordering $(a_i)_{i\ge 0}$ 
  \begin{equation*}
    \val_P\!\left(\fixdiv\left(\prod_{i=0}^{m}(x-a_i)\right)\right) = \alpha_P(m)
  \end{equation*}
  holds. This can be seen by the following argument. Any choice of $m$
  consecutive elements of $(a_i)_{i\ge 0}$, for simplicity say $a_i$
  with $0\le i \le m-1$, contains exactly
  $\left\lfloor \frac{m}{q^j} \right \rfloor$ complete systems of
  residues modulo $q^j$ for all $j\ge 1$ (where $\lfloor\;.\;\rfloor$
  denotes the floor operator on $\QQ$).

  Note that $a_m$ is congruent to exactly
  $\left\lfloor \frac{m}{q^{j}}\right\rfloor$ elements modulo $P^j$
  for $j\ge 0$. In addition, $\max_{0 \le i \le m-1}\val(a_m-a_i) = n$, where
  $n\in \N$ with $q^n \le m < q^{n+1}$. Therefore,
  \begin{align*}
    \alpha_P(m)
    &=  \val_P\!\left(\prod_{i=0}^{m-1}(a_m - a_i)\right)
     = \sum_{j=1}^n |\{1\le i \le m-1 \mid \val(a_m - a_i) \ge j \}| \\
    &= \sum_{j\ge 1}\left\lfloor \frac{m}{q^j} \right \rfloor
     = \sum_{j=0}^{n-1}q^j
     = \alpha_P(q^n).
  \end{align*}
\end{remark}

\begin{definition}\label{def:gen-binom-poly}
  Let $D$ be a Dedekind domain and assume that $k\ge 1$ is an integer
  such that $\genfac_D(k) = cD$ is a principal ideal and $\calP$ be
  the (finite) set of prime ideals which contain $c$. Further, let
  $(a_i)_{i=0}^{k-1}$ be a sequence in $D$ such that, for each
  $P \in \calP$, $(a_i)_{i=0}^{k-1}$ is an initial sequence for a
  $P$-ordering. We call
  \begin{equation*}
   c^{-1} \prod_{i=0}^{k-1}(x-a_i)
  \end{equation*}
  a \emph{generalized binomial polynomial of degree $k$ over $D$}.
\end{definition}

We now turn our attention to generalized binomial polynomials of
degree $q^n$ where $q$ is the (finite) order of a maximal ideal $P$ of
a Dedekind domain and $n\in \N$.  Let $S$ be a choice of $q^n$
consecutive elements of a $P$-ordering $(a_i)_{i\ge 0}$ of $D$. It
follows from Remark~\ref{remark:p-ordering-properties} that $S$ is a
system of representatives of the residue classes of $P^n$. Therefore,
considered as subset of the discrete valuation domain $R = D_P$ with
maximal ideal $M = PD_P$, $S$ is a $(D_P, P_P)$-balanced set. By
Remark~\ref{rem:csr-equpoly}, its equalizing polynomial is
$\prod_{s\in S}(x-s)$. (Note that a segment of consecutive elements of
a $P$-ordering whose length is not a power of $q$ is not a balanced
set.)

In view of this discussion, the following assertion now follows from
Corollary~\ref{corollary:globalizing}.

\begin{corollary}\label{corollary:gen-binom}
  Let $D$ be a Dedekind domain and $P$ be a prime ideal of $D$ with
  finite index $q$ and let $n\in \N$ such that the generalized
  factorial $\genfac_D(q^n) = (c)$ is a principal ideal of $D$.

  Then, for every $P$-ordering $(a_i)_{i\ge 0}$, the polynomial
  \begin{equation*}
    c^{-1}\prod_{i=0}^{q^n-1}(x-a_i)
  \end{equation*}
  is absolutely irreducible in $\Int(D)$.
\end{corollary}

\begin{remark}
  Let $p$ be an integer prime number and $n\in \N$.
  \begin{enumerate}
  \item It follows by Corollary~\ref{corollary:gen-binom} that the
    (classical) binomial polynomial $\binom{x}{p^n}$ is absolutely
    irreducible in $\Int(\Z)$. Note that this is a special case of a
    result of Rissner and
    Windisch~\cite[Theorem~1]{Rissner-Windisch:2021:binom}, who have shown
    that $\binom{x}{m}$ is absolutely irreducible in $\Int(\Z)$ for
    all $m\in \N$.
  \item A new result which immediately follows from by
    Corollary~\ref{corollary:gen-binom} is that $\binom{x}{p^n}$ is
    absolutely irreducible in $\Int(\Z_{(p)})$ as well as in
    $\Int(\Z_{p})$, where $\Z_{(p)}$ denotes the localization of
    $\Z$ at $p$ and $\Z_{p}$ the $p$-adic integers.
  \end{enumerate}
\end{remark}

\frenchspacing
\bibliographystyle{amsplainurl}
\bibliography{biblio}

\providecommand{\bysame}{\leavevmode\hbox to3em{\hrulefill}\thinspace}
\providecommand{\MR}{\relax\ifhmode\unskip\space\fi MR }
\providecommand{\MRhref}[2]{%
  \href{http://www.ams.org/mathscinet-getitem?mr=#1}{#2}
}
\providecommand{\href}[2]{#2}
\begin{thebibliography}{10}

\bibitem{AndersonS-Cahen-Chapman-Smith:1995:fac-iv}
David~F. Anderson, Paul-Jean Cahen, Scott~T. Chapman, and William~W. Smith,
  \emph{Some factorization properties of the ring of integer-valued
  polynomials}, Zero-dimensional commutative rings ({K}noxville, {TN}, 1994),
  Lecture Notes in Pure and Appl. Math., vol. 171, Dekker, New York, 1995,
  pp.~125--142. \MR{1335709}

\bibitem{BM1997:P-orderings}
Manjul Bhargava,
  \href{http://dx.doi.org/10.1515/crll.1997.490.101}{\emph{{$P$}-orderings and
  polynomial functions on arbitrary subsets of {D}edekind rings}}, J. Reine
  Angew. Math. \textbf{490} (1997), 101--127. \MR{1468927}

\bibitem{Cahen-Chabert:1995:Elasticity-for-IVP}
Paul-Jean Cahen and Jean-Luc Chabert,
  \href{http://dx.doi.org/10.1016/0022-4049(94)00108-U}{\emph{Elasticity for
  integral-valued polynomials}}, J. Pure Appl. Algebra \textbf{103} (1995),
  no.~3, 303--311. \MR{1357791}

\bibitem{Cahen-Chabert:1997:book}
\bysame, \emph{Integer-valued polynomials}, Mathematical Surveys and
  Monographs, vol.~48, American Mathematical Society, Providence, RI, 1997.
  \MR{1421321}

\bibitem{Cahen-Chabert:2016:survey}
\bysame,
  \href{http://dx.doi.org/10.4169/amer.math.monthly.123.4.311}{\emph{What you
  should know about integer-valued polynomials}}, Amer. Math. Monthly
  \textbf{123} (2016), no.~4, 311--337. \MR{3493376}

\bibitem{CahenChabertFrisch:2000:interpolation}
Paul-Jean Cahen, Jean-Luc Chabert, and Sophie Frisch,
  \href{http://dx.doi.org/10.1006/jabr.1999.8151}{\emph{Interpolation
  domains}}, J. Algebra \textbf{225} (2000), no.~2, 794--803. \MR{1741562}

\bibitem{ChKr2012:Atomic-decay}
Scott~T. Chapman and Ulrich Krause,
  \href{https://www.degruyter.com/view/book/9783110278606/10.1515/9783110278606.301.xml}{\emph{A
  closer look at non-unique factorization via atomic decay and strong atoms}},
  Progress in commutative algebra 2, Walter de Gruyter, Berlin, 2012,
  pp.~301--315. \MR{2932599}

\bibitem{Chapman-McClain:2005:irred-iv-poly}
Scott~T. Chapman and Barbara~A. McClain,
  \href{http://dx.doi.org/10.1016/j.jalgebra.2005.01.026}{\emph{Irreducible
  polynomials and full elasticity in rings of integer-valued polynomials}}, J.
  Algebra \textbf{293} (2005), no.~2, 595--610. \MR{2173716}

\bibitem{Frisch:2014:monadic}
Sophie Frisch,
  \href{http://dx.doi.org/10.1007/978-3-319-38855-7_6}{\emph{Relative
  polynomial closure and monadically {K}rull monoids of integer-valued
  polynomials}}, Multiplicative ideal theory and factorization theory, Springer
  Proc. Math. Stat., vol. 170, Springer, [Cham], 2016, pp.~145--157.
  \MR{3565807}

\bibitem{FrNa2019:Graphtheoretic}
Sophie Frisch and Sarah Nakato,
  \href{http://dx.doi.org/10.1080/00927872.2020.1744618}{\emph{A
  graph-theoretic criterion for absolute irreducibility of integer-valued
  polynomials with square-free denominator}}, Communications in Algebra
  \textbf{0} (2020), no.~0, 1--8.

\bibitem{FriPaTiWi:1999:additive-measures}
Sophie Frisch, Milan Pa\v{s}t\'{e}ka, Robert~F. Tichy, and Reinhard Winkler,
  \href{http://dx.doi.org/10.1007/BF02857307}{\emph{Finitely additive measures
  on groups and rings}}, Rend. Circ. Mat. Palermo (2) \textbf{48} (1999),
  no.~2, 323--340. \MR{1692930}

\bibitem{GeHa2006:NUFactorizations}
Alfred Geroldinger and Franz Halter-Koch,
  \href{http://dx.doi.org/10.1201/9781420003208}{\emph{Non-unique
  factorizations}}, Pure and Applied Mathematics (Boca Raton), vol. 278,
  Chapman \& Hall/CRC, Boca Raton, FL, 2006, Algebraic, combinatorial and
  analytic theory. \MR{2194494}

\bibitem{Kaczorowski:1981:compl-irred}
Jerzy Kaczorowski, \emph{Completely irreducible numbers in algebraic number
  fields}, Funct. Approx. Comment. Math. \textbf{11} (1981), 95--104.
  \MR{692718}

\bibitem{Loper:1998:intdpruefer}
K.~Alan Loper, \href{http://dx.doi.org/10.1090/S0002-9939-98-04459-1}{\emph{A
  classification of all {$D$} such that {${\rm Int}(D)$} is a {P}r\"ufer
  domain}}, Proc. Amer. Math. Soc. \textbf{126} (1998), no.~3, 657--660.
  \MR{1459137}

\bibitem{McClain:2004:honorsthesis}
Barbara~Anne McClain,
  \href{https://citeseerx.ist.psu.edu/viewdoc/download?doi=10.1.1.514.2855&rep=rep1&type=pdf}{\emph{Factorisation
  properties of integer-valued polynomials}}, Honors thesis at {T}rinity
  {U}niversity, 2004.

\bibitem{NS2020:Non-Abs}
Sarah Nakato,
  \href{http://dx.doi.org/10.1080/00927872.2019.1705474}{\emph{Non-absolutely
  irreducible elements in the ring of integer-valued polynomials}},
  Communications in Algebra \textbf{48} (2020), no.~4, 1789--1802.

\bibitem{Ostrowski:1919:regb}
Alexander Ostrowski,
  \href{http://dx.doi.org/10.1515/crll.1919.149.117}{\emph{\"{U}ber ganzwertige
  {P}olynome in algebraischen {Z}ahlk\"{o}rpern}}, J. Reine Angew. Math.
  \textbf{149} (1919), 117--124. \MR{1580967}

\bibitem{Polya:1919:regb}
Georg P\'{o}lya,
  \href{http://dx.doi.org/10.1515/crll.1919.149.97}{\emph{\"{U}ber ganzwertige
  {P}olynome in algebraischen {Z}ahlk\"{o}rpern}}, J. Reine Angew. Math.
  \textbf{149} (1919), 97--116. \MR{1580966}

\bibitem{Reinhart:2014:monadic}
Andreas Reinhart, \emph{On monoids and domains whose monadic submonoids are
  {K}rull}, Commutative algebra, Springer, New York, 2014, pp.~307--330.
  \MR{3330226}

\bibitem{Rissner-Windisch:2021:binom}
Roswitha Rissner and Daniel Windisch,
  \href{http://dx.doi.org/https://doi.org/10.1016/j.jalgebra.2021.03.007}{\emph{Absolute
  irreducibility of the binomial polynomials}}, Journal of Algebra \textbf{578}
  (2021), 92--114.

\end{thebibliography}

\end{document}